\documentclass[12pt]{amsart}
\usepackage{amssymb,amsmath}
\usepackage[dvips]{graphicx}
\usepackage{pslatex}
\usepackage{amsmath,amscd}
\usepackage{amsthm}

\theoremstyle{plain} \newtheorem{thm}{Theorem}[section]
\newtheorem{cor}[thm]{Corollary} \newtheorem{prop}[thm]{Proposition}
\newtheorem{lemma}[thm]{Lemma} 

\newtheorem{question}[thm]{Question} 

\newtheorem*{namedtheorem}{\theoremname}
\newcommand{\theoremname}{testing}

\theoremstyle{definition} \newtheorem{defn}[thm]{Definition}

\theoremstyle{remark}

    \DeclareMathOperator{\sech}{sech}

\begin{document}

\title{Mutations and short geodesics in hyperbolic $3$-manifolds}

\author{Christian Millichap} 
\address{Department of Mathematics, Linfield College\\ McMinnville, OR 97128}
\email{Cmillich@linfield.edu}

\begin{abstract}
In this paper, we explicitly construct large classes of incommensurable hyperbolic knot complements with the same volume and the same initial (complex) length spectrum. Furthermore, we show that these knot complements are the only knot complements in their respective commensurability classes by analyzing their cusp shapes.

The knot complements in each class differ by a topological cut-and-paste operation known as mutation. Ruberman has shown that mutations of hyperelliptic surfaces inside hyperbolic $3$-manifolds preserve volume.  Here, we provide geometric and topological conditions under which such mutations also preserve the initial (complex) length spectrum.  This work requires us to analyze when least area surfaces could intersect short geodesics in a hyperbolic $3$-manifold. 
\end{abstract}

\maketitle
\section{Introduction}
\label{sec:intro}

The work of Mostow and Prasad implies that every finite volume hyperbolic $3$-manifold admits a unique hyperbolic structure, up to isometry \cite{Pr},\cite{Mos}.  Thus, geometric invariants of a hyperbolic manifold, such as volume and geodesic lengths, are also topological invariants. It is natural to ask: how effective can such invariants be at distinguishing hyperbolic $3$-manifolds?  Furthermore, how do these invariants interact with one another?

In this paper, we will study how mutations along \emph{hyperelliptic surfaces} inside of a hyperbolic $3$-manifold affect such invariants. A hyperelliptic surface $F$ is a surface admitting a \emph{hyperelliptic involution}: an order two automorphism of $F$ which fixes every isotopy class of curves in $F$. A \emph{mutation} along a hyperelliptic surface $F$ in a hyperbolic $3$-manifold $M$ is an operation where we cut $M$ along $F$, and then reglue by a hyperelliptic involution $\mu$ of $F$, often producing a new $3$-manifold, $M^{\mu}$. While a mutation can often change the global topology of a manifold, the action is subtle enough that many geometric, quantum, and classical invariants are preserved under mutation; see \cite{DGST} for details.  In particular, Ruberman showed that mutating hyperbolic $3$-manifolds along incompressible, $\partial$-incompressible surfaces preserves hyperbolicity and volume in \cite{Ru}.

Here, we investigate under which conditions such mutations preserve the smallest $n$ values of the length spectrum, the \emph{initial length spectrum}. The \emph{length spectrum} of a manifold, $M$, is the set of all lengths of closed geodesics in $M$ counted with multiplicites.  We will also consider the \emph{complex length spectrum} of $M$: the set of all complex lengths of closed geodesics in $M$ counted with multiplicities.  Given a closed geodesic $\gamma \subset M$, the \emph{complex length} of $\gamma$ is the complex number $\ell_{\mathbb{C}}(\gamma) = \ell(\gamma) + i \theta$ where $\ell(\gamma)$ denotes the length of $\gamma$ and $\theta$ is the angle of rotation incurred by traveling once around $\gamma$.

Throughout this paper, any surface will be connected, orientable, and of finite complexity, unless stated otherwise. Any hyperbolic $3$-manifold $M$ will have finite volume and be connected, complete, and orientable. Our investigation requires a surface that we mutate along to be a \emph{least area surface in $M$}, or a close variant, to be defined later.

\begin{defn}[Least Area Surface in $M$]
\label{defn:LA}
Let $F \subset M$ be a properly and smoothly embedded surface in a Riemannian $3$-manifold $M$. Then $F$ is called a \emph{least area surface} if $F$ minimizes area in its homotopy class.  
\end{defn}

Least area surfaces inside of $3$-manifolds are well studied objects. Schoen--Yau \cite{ScYa} showed that incompressible surfaces inside closed $3$-manifolds can always be homotoped to smoothly immersed least area surfaces. Freedman--Hass-- Scott \cite{FHS} showed that this resulting immersion is an embedding. Ruberman expanded this analysis to noncompact surfaces in noncompact hyperbolic $3$-manifolds in \cite{Ru}, where he provided conditions for the existence, uniqueness, and embeddedness of least area surfaces in a hyperbolic $3$-manifold.

The following theorem gives three possible properties of a hyperbolic $3$-manifold $M$ that can help determine the topology amd geometry of $\gamma \cap F \subset M$, where $\gamma$ is a closed geodesic and $F$ is an incompressible surface. These properties are the maximal embedded tube radius $r$ of a neighborhood of $\gamma$, denoted $T_{r}(\gamma)$, the length of $\gamma$, denoted $\ell(\gamma)$, and the \emph{normalized length} of a Dehn filling, which we describe in Definition \ref{defn:NL}.

By a closed curve $n \cdot \gamma$, we mean a simple closed curve that is in the homotopy class of $[n \cdot \gamma] \in \pi_{1}(\partial T_{r}(\gamma))$. We can now state this result.

\begin{thm}
\label{thm:main}
Let $M$ be a hyperbolic manifold with $F \subset M$ an embedded surface that is incompressible and $\partial$-incompressible with $\left|\chi(F) \right| \leq 2$. Let $\gamma \subset M$ be a closed geodesic with embedded tubular radius $r$. Assume
\begin{enumerate}
\item $r > 2 \ln(1 + \sqrt{2}) $, or  
\item $\ell(\gamma) < 0.015$, or
\item $\gamma$ is the core of a solid torus added by Dehn filling $N \cong M \setminus \gamma$ along a slope of normalized length $\widehat{L} \geq 14.90$.
\end{enumerate}

Then $\gamma$ can be isotoped disjoint from $F$. Furthermore, if $F$ is embedded in least area form, then either $\gamma \cap F = \emptyset$ without any isotopy or $n \cdot \gamma$ is isotopic into $F$ for some $n \in \mathbb{N}$.

\end{thm}

A few remarks about this theorem:  
\begin{itemize}
\item This theorem is both a topological and a geometrical statement about $\gamma \cap F$. Only the topological statement is necessary for showing that mutation preserves the initial length spectrum; see Theorem \ref{cor:syspreserved}.
\item This theorem is stated in full generality in Theorem \ref{thm:gammasep} where no constraints are made on the Euler characteristic. We mainly care about $\left|\chi(F) \right| \leq 2$ because the surfaces we will consider in our main result (Theorem \ref{cor:systole_and_vol}) are all\textit{ Conway spheres}, i.e, $4$-punctured spheres inside of knot complements.  
\item Theorem \ref{thm:gammasep} is stated in terms of \textit{almost least area surfaces}, which generalize least area surfaces; see Definition \ref{def:ALAS}. 
\item $(2)$ implies $(1)$ by the work of Meyerhoff stated in Theorem \ref{thm:Collar_lemma}. $(3)$ implies $(1)$ by the work of Hodgson and Kerckhoff \cite{HoKe}, \cite{HoKe2} on cone deformations. Furthermore, a version of $(3)$ implies $(2)$ exists, but we must adjust the lower bound on normalized length to be $\widehat{L} \geq 20.76$.
\item $(3)$ can be stated in terms of Dehn filling multiple curves; see Corollary \ref{cor:disjointgeo}.
\end{itemize}

The proof of Theorem \ref{thm:main} relies on both the topology and geometry of $F \cap T_{r}(\gamma)$, where $T_{r}(\gamma)$ is the embedded tubular neighborhood of radius $r$ around $\gamma$.  Since $F$ is incompressible and $\partial$-incompressible, $F$ can be isotoped into almost least area form by Theorem \ref{thm:LAsurfaces}. As a result, components of $F \cap T_{r}(\gamma)$ must be disks or annuli.  If a component of $F \cap T_{r}(\gamma)$ that intersects $\gamma$ is a disk, $D_{r}$, then we work to get an area contradiction.  Specifically, if $r$ is sufficiently large, then the area of $D_{r}$ inside of this neighborhood will be too big, and so, $\gamma$ must be disjoint from $F$ in this case.  As mentioned in the remarks, conditions $(2)$ and $(3)$ each imply $(1)$, so all of our cases rely on a sufficiently large tube radius in the end.  If a component of  $F \cap T_{r}(\gamma)$ that intersects $\gamma$ is an annulus, $A_{r}$, then this annulus must be parallel to the boundary torus $\partial T_{r}(\gamma)$.  Here, $\gamma$ can be isotoped disjoint from $A_{r}$, and more generally, isotoped disjoint from $F$.    

The following theorem tells us when the initial length spectrum is preserved under mutation. 

\begin{thm}
\label{cor:syspreserved}
Let $F \subset M$ be a properly embedded surface that is incompressible, $\partial$-incompressible, and admits a hyperelliptic involution $\mu$. Suppose that $M$ has $n$ geodesics shorter than some constant $L <0.015$. Then $M$ and $M^{\mu}$ have (at least) the same $n$ initial values of their respective (complex) length spectra.   
\end{thm}

Under these hypotheses, any sufficiently short geodesic $\gamma$ in $M$ can be isotoped disjoint from $F$.  After this isotopy, if we mutate $M$ along $(F, \mu)$ to obtain $M^{\mu}$, then there will also be a closed curve in $M^{\mu}$ corresponding with $\gamma$.  We just need to analyze the representations $\rho: \pi_{1}(M) \rightarrow \text{PSL}(2, \mathbb{C})$ and $\rho_{\mu}: \pi_{1}(M^{\mu}) \rightarrow \text{PSL}(2, \mathbb{C})$ to see that $[\gamma]$, as an element of either $\pi_{1}(M)$ or $\pi_{1}(M^{\mu})$, has the same representation (up to conjugacy) in $\text{PSL}(2, \mathbb{C})$, and so, the same (complex) length associated to it in either case. Note that Theorem \ref{cor:syspreserved} only relies on the topological statement from Theorem \ref{thm:main}. In fact, any $\gamma$ that can be homotoped disjoint from $F$ will be preserved under mutation since we only need to consider $\gamma$ as a representative of an element of $\pi_{1}(M)$ or $\pi_{1}(M^{\mu})$; this follows from Theorem \ref{thm:mutationrep} and Lemma \ref{lemma:Fgroups}.   

This theorem gives us a tool to produce non-isometric hyperbolic $3$-manifolds that have at least the same initial length spectrum. Over the past $35$ years, there have been a number of constructions for producing non-isometric hyperbolic $3$-manifolds that are \textit{iso-length spectral}, i.e., have the same length spectrum. Vign\'{e}ras in \cite{Vi} used arithmetic techniques to produce the first known constructions of such manifolds.  Sunada developed a general method for constructing iso-length spectral manifolds \cite{Su}, which helped him produce many iso-length spectral, non-isometric Riemann surfaces.  This technique produces covers of a manifold $M$ that are iso-length spectral by finding certain group theoretic conditions on subgroups of $\pi_{1}(M)$. We will refer to any such group theoretic construction for producing covers that have either the same length spectrum or some variation of this as a \textit{Sunada-type construction}. 

Since Sunada's original work, many Sunada-type constructions have been developed.  These constructions often have very interesting relations to volume. McReynolds uses a Sunada-type construction in \cite{McR} to build arbitrarily large sets of closed, iso-length spectral, non-isometric hyperbolic manifolds.  Furthermore, the growth of size of these sets of manifolds as a function of volume is super-polynomial.  In contrast, Leininger--McReynolds--Neumann--Reid in \cite{LMNR}  also use a Sunada-type construction to show that for any closed hyperbolic $3$-manifold $M$, there exists infinitely many covers $\left\{M_{j}, N_{j}\right\}$ of $M$, such that the length sets of these pairs are equal but $\frac{vol(M_{j})}{vol(N_{J})} \rightarrow \infty$.  Here, the \textit{length set} of a manifold is the set of all lengths of closed geodesics counted without multiplicities.  Thus, volume can behave drastically differently for hyperbolic $3$-manifolds that are iso-length spectral as compared with hyperbolic $3$-manifolds with the same length set.  

All of the constructions mentioned above produce \textit{commensurable manifolds}, that is, manifolds that share a common finite-sheeted cover. Sunada type constructions will always produce commensurable manifolds since they involve taking covers of a common manifold and commensurability is an equivalence relation. On the other hand, the work of Reid \cite{Re} and Chinburg--Hamilton--Long--Reid \cite{ChHaLoRe} shows that iso-length spectral, non-isometric \underline{arithmetic} hyperbolic $3$-manifolds are \textit{always} commensurable.  To date, all known examples of iso-length spectral, non-isometric hyperbolic $3$-manifolds are commensurable. This raises the following question:

\begin{question}
\label{q:spectral}
Do there exist incommensurable iso-length spectral hyperbolic $3$-manifolds?
\end{question}

Here, we construct large families of mutant pretzel knot complements which have the same initial (complex) length spectrum, the same volume, and are pairwise incommensurable. Our construction does not use arithmetic methods or a Sunada-type construction, but rather, the simple cut and paste operation of mutating along Conway spheres. This work is highlighted in our main theorem below. See Section \ref{sec:RT_and_PK} for the definition of a pretzel knot.

\begin{thm}
\label{cor:systole_and_vol}
For each $n \in \mathbb{N}$, $n \geq 2$, there exist $\frac{(2n)!}{2}$ non-isometric hyperbolic pretzel knot complements that differ by mutation, $\left\{M_{2n+1}^{\sigma}\right\}$, such that these manifolds:

\begin{itemize}
\item have the same $2n+1$ shortest geodesic (complex) lengths, 
\item are pairwise incommensurable,
\item have the same volume, and
\item $\left(\frac{2n-1}{2}\right)v_{\mathrm{oct}} \leq vol(M^{\sigma}_{2n+1})  \leq \left(4n+2\right)v_{\mathrm{oct}}$, where $v_{\mathrm{oct}} \left(\approx 3.6638\right)$ is the volume of a regular ideal octahedron.  

\end{itemize}
 
\end{thm}

Theorem \ref{cor:systole_and_vol} provides an answer to a weak form of Question \ref{q:spectral}. While these mutant pretzel knot complements have the same initial length spectrum, we doubt that any of them are actually iso-length spectral. Almost all sufficiently long geodesics in one of these pretzel knot complements have homotopically essential intersections with all of the Conway spheres. Thus, their corresponding geodesic lengths should be changed by mutation.

The fact that these hyperbolic pretzel knot complements are pairwise incommensurable comes from the following theorem. See Section \ref{sec:commensurablity} for full details. 

\begin{thm}
\label{thm:incom}
Let $n \geq 2$ and let $q_{1}, \ldots, q_{2n+1}$ be integers such that only $q_{1}$ is even, $q_{i} \neq q_{j}$ for $i \neq j$, and all $q_{i}$ are sufficiently large. Then the complement of the hyperbolic pretzel knot  $K \left( \frac{1}{q_{1}}, \frac{1}{q_{2}}, \ldots, \frac{1}{q_{2n+1}} \right)$  is the only knot complement in its commensurability class.  In particular, any two of these hyperbolic pretzel knot complements are incommensurable.
\end{thm}

Proving that a particular knot complement is the only knot complement in its commensurability class is generally not an easy task.  Only two large classes of knot complements are known to have this property. Reid and Walsh in \cite{ReWa} have shown that hyperbolic $2$-bridge knot complements are the only knot complements in their respective commensurability classes, and similarly, Macasieb and Mattman in \cite{MM} have shown this for the complements of hyperbolic pretzel knots of the form $K\left( \frac{1}{-2}, \frac{1}{3}, \frac{1}{n} \right)$, $n \in \mathbb{Z} \setminus \left\{7\right\}$.  Usually the hardest part of this work is showing that these knot complements have no \textit{hidden symmetries}, that is, these knot complements are not irregular covers of orbifolds. We are able to rule out hidden symmetries by analyzing the cusp shapes of certain \textit{untwisted augmented links} (see Section \ref{sec:GEOofUAL})  that we Dehn fill along to obtain our pretzel knot complements.

Now, let us outline the rest of this paper.  In Section \ref{subsec:MR}, we prove the monotonicity of the mass ratio for least area disks in $\mathbb{H}^{3}$.  This result helps give a lower bound on the area of a least area disk inside a ball in $\mathbb{H}^{3}$. Section \ref{sec:LA_surfaces_and_geo} gives the proof of Theorem \ref{thm:main} and states this result in its full generality. This section is broken down into subsections, each dealing with one of the conditions to be satisfied for Theorem \ref{thm:main}.  Section \ref{subsec:symmsurf_and_mut} gives the proof of Theorem \ref{cor:syspreserved} and a number of corollaries to this theorem. In Section \ref{sec:RT_and_PK}, we construct and describe our class of hyperbolic pretzel knots which are mutants of one another.  We also highlight a theorem from our past work \cite{Mi} that describes how many of these mutant pretzel knot complements are non-isometric and have the same volume. In Section \ref{sec:GEOofUAL}, we analyze the geometry of our pretzel knots by realizing them as Dehn fillings of untwisted augmented links, whose complements have a very simple polyhedral decomposition. In particular, this analysis allows us to put a lower bound on the normalized lengths of the Dehn fillings performed to obtain our pretzel knot complements, and also, helps determine the cusp shapes of the pretzel knots themselves. In Section \ref{sec:commensurablity} we prove that these knots are pairwise incommensurable. In Section \ref{sec:mutations_sys}, we apply Theorem \ref{cor:syspreserved} to show that our class of pretzel knot complements have the same initial length spectrum.  We also give an application to closed hyperbolic $3$-manifolds with the same initial length spectrum. Putting all these results together gives Theorem \ref{cor:systole_and_vol} in Section \ref{sec:mutations_sys}. 

We are grateful to David Futer for his help and guidance with this project.  We thank Frank Morgan for directing us to the monotonicity of the mass ratio result found in his book \cite{Mo}.  We thank Jessica Purcell for providing useful comments and help with understanding cone deformations. Finally, we thank the referees for making numerous helpful comments. 

%%%%%%%%%%%%%%%%%%%%%%%%%%%%%%%%%%%%%%%%%%%%%%%%%%%%%%%%%%%

\section{Monotonicity of the Mass ratio for least area disks in $\mathbb{H}^{3}$}
\label{subsec:MR}

Throughout this section, $\ell(-)$ will denote hyperbolic length, and $B(a,r) \subset \mathbb{H}^{3}$ will denote a ball of radius $r$ centered at $a$.  Also, $A(-)$ will denote the area that a smoothly immersed surface inherits from a hyperbolic $3$-manifold by pulling back the hyperbolic metric. Here, we establish a useful result for least area disks in $\mathbb{H}^{3}$.  

\begin{defn}[Least Area Disk]
\label{defn:LA2}
Let $D \subset M$ be a properly and smoothly embedded disk in a Riemannian $3$-manifold $M$.  Let $c$ be a simple closed curve in $M$ such  that $\partial D =c$.  Then $D$ is called a \emph{least area disk} in $M$, if $D$ minimizes area amongst all properly and smoothly immersed disks with boundary $c$. 
\end{defn}

The compactness theorem in \cite[Theorem 5.5]{Mo} guarantees that this infimum is always realized for disks in $\mathbb{R}^{n}$. Furthermore, the regularity theorem in \cite[Theorem 8.1]{Mo} says such an area minimizing disk is smooth and embedded in its interior. Similar results hold for disks in $\mathbb{H}^{n}$. The following definition will be useful for analyzing least area disks in $\mathbb{H}^{3}$.

\begin{defn}[Mass Ratio and Density]
Let $a \in \mathbb{H}^{3}$ and consider $A(D \cap B(a,r))$, the area of a disk inside a ball.  Define the \emph{mass ratio} to be 
\begin{center}
$\Theta (D, a, r) = \frac{A(D \cap B(a,r))}{4\pi\sinh^{2}(\frac{r}{2})}$.
\end{center} 
Define the \emph{density} of $D$ at $a$ to be
\begin{center}
$\Theta (D,a) = \lim_{r \rightarrow 0} \Theta (D, a, r)$.
\end{center}
\end{defn}

A few comments about the above definition. First, $4\pi\sinh^{2}(\frac{r}{2})$ is the area of a totally geodesic disk of radius $r$ in $\mathbb{H}^{n}$.  Also, for smoothly immersed surfaces, $\Theta (D,a) \geq 1$ at any point $a \in D$. For an embedded surface we actually have $\Theta (D,a) = 1$. If $D$ is not embedded at a point $a \in D$, then restricting to a subset of $D'$ of $D$ so that $D' \cap B(a,r)$ is an embedding only decreases the numerator of the mass ratio. See \cite[Chapter 2]{Mo} for more on densities.

The monotonicity of the mass ratio was proved in the case for Euclidean geometry by Federer \cite{Fe} and a proof can also be found in Morgan \cite[Theorem $9.3$]{Mo}.  Here, we obtain a similar result in $\mathbb{H}^{3}$ by using the same techniques as the proof given in Morgan. Also, this result is proved in greater generality in \cite[Section 2]{An}.

\begin{thm}
\label{thm:Monotonicity_of_MR}
Let $D$ be a least area disk in $\mathbb{H}^{3}$. Let $a \in \mathring D \subset \mathbb{H}^{3}$.  Then for $0 < r< d(a, \partial D)$, the mass ratio $\Theta (D, a, r)$ is a monotonically increasing function of $r$. 
\end{thm}

To prove this theorem, we need the following basic fact in hyperbolic trigonometry:

\begin{lemma}
\label{lemma:hyptrig}
$\frac{\sinh(\frac{r}{2})}{\cosh(\frac{r}{2})} = \frac{\cosh(r)-1}{\sinh(r)}$, for $r >0$.
\end{lemma}

\begin{proof}

This is a simple algebraic exercise that requires a few identities:

\begin{center}
$\frac{\sinh(\frac{r}{2})}{\cosh(\frac{r}{2})} = \sqrt{\frac{\cosh(r)-1}{\cosh(r)+1}} = \frac{\cosh(r)-1}{\sqrt{\cosh^{2}(r)-1}} = \frac{\cosh(r)-1}{\sinh(r)}$.
\end{center}

The first equality comes from well-known half-angle formulas. The rest of the equalities come from algebraic manipulations and the fact that $1 = \cosh^{2}(r) - \sinh^{2}(r)$.
\end{proof}

 \begin{proof}[Proof of Theorem \ref{thm:Monotonicity_of_MR}]
For $0 < r< d(a, \partial D)$, let $f(r)$ denote $A(D \cap B(a,r))$.  Obviously, $f$ is monotonically increasing, which implies that $f'(r)$ exists almost everywhere. Set $\gamma_{r} = \partial (D \cap B(a,r))$. Now, we have that

\begin{center}
 $(1)$ \; \;  $ \ell(\gamma_{r}) \leq f'(r)$,
\end{center}
 which is the ``co-area formula'' from \cite[Lemma 2.2]{HS}.  This inequality holds whenever $\gamma_{r}$ is a $1$-manifold, i.e., whenever $D$ intersects $\partial B(a,r)$ transversely. Since $D$ is area-minimizing, $A(D \cap B(a,r)) \leq A(C)$, where $C$ is the cone over $\gamma_{r}$ to $a$.

\begin{figure}[ht]
\includegraphics[scale=0.65]{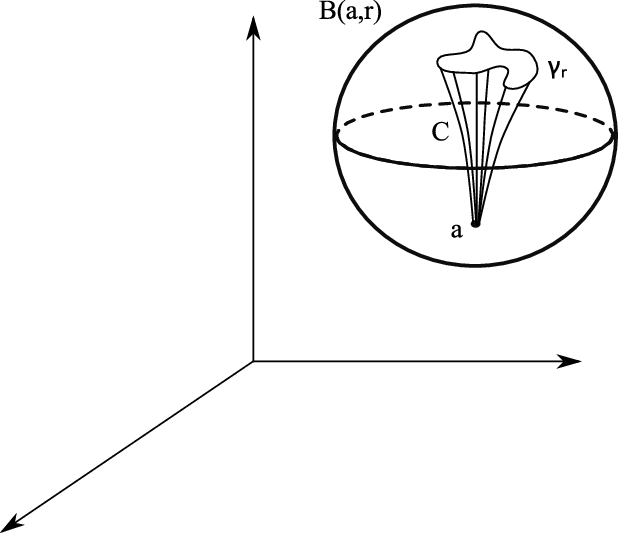}
\caption{The hyperbolic cone $C$ over $\gamma_{r}$ to $a$ in the upper half-space model of $\mathbb{H}^{3}$.}
\label{Cone_in_Ball}
\end{figure}

\textbf{Claim:} $A(C) = \ell(\gamma_{r}) \frac{\cosh(r) - 1}{\sinh(r)}$.
\\
Let $\gamma$ be the projection of $\gamma_{r}$ to the unit tangent sphere centered at $a$.  Our area form is $dA =  ds dR$, where $dR$ is the change in radius of a hyperbolic sphere and $ds = \sinh(R) d\theta$ is arc length on a sphere of radius $R$.  The area form on $A(C)$ is inherited from geodesic polar coordinates in $\mathbb{H}^{3}$. We have that 
\begin{center}
$A(C) = \int_{0}^{r} \int_{0}^{\ell(\gamma_{R})} ds dR =  \int_{0}^{r} \int_{0}^{\ell(\gamma_{R})} \sinh(R) d\theta dR = \int_{\gamma} d\theta \int_{0}^{r} \sinh(R) dR = \ell(\gamma)(\cosh(r) - 1)$.
\end{center}
In order to rescale to make $A(C)$ a function of $\ell(\gamma_{r})$, we use the fact that $\ell(\gamma_{r}) = \int_{\gamma}\sinh(r) d\theta = \ell(\gamma) \sinh(r)$ to get that $A(C) = \ell(\gamma_{r}) \frac{\cosh(r)-1}{\sinh(r)}$.  
\\
Putting (1) together with the previous claim and Lemma \ref{lemma:hyptrig} gives:
\begin{center}
 $f(r) \leq A(C) = \ell(\gamma_{r}) \frac{\cosh(r) - 1}{\sinh(r)} \leq f'(r) \frac{\cosh(r) - 1}{\sinh(r)} = f'(r) \frac{\sinh(\frac{r}{2})}{\cosh(\frac{r}{2})}$.
\end{center} 
Consequently,
\begin{center}
$\frac{d}{dr} \left[4\pi \Theta (D, a, r)\right] = \frac{d}{dr}\left[f(r) \sinh^{-2}(\frac{r}{2}) \right] = \frac{f'(r)}{\sinh^{2}(\frac{r}{2})} - \frac{f(r)\cosh(\frac{r}{2})}{\sinh^{3}(\frac{r}{2})}  =  \frac{\cosh(\frac{r}{2})}{\sinh^{3}(\frac{r}{2})} \left[ f'(r)\frac{\sinh(\frac{r}{2})}{\cosh(\frac{r}{2})} - f(r) \right] \geq 0$ 
\end{center}
since $\frac{\cosh(\frac{r}{2})}{\sinh^{3}(\frac{r}{2})} \geq 0$ for any $r > 0$.  
\end{proof}

The following corollary will play a pivotal role in Section \ref{sec:LA_surfaces_and_geo}.

\begin{cor}
\label{cor2}
Suppose $D \subset \mathbb{H}^{3}$ is a least area disk and $a \in \mathring D$. Then $A(D \cap B(a,r)) \geq  4\pi\sinh^{2}(\frac{r}{2})$, for any $r$, $0 < r \leq d(a, \partial D)$. 
\end{cor}

\begin{proof}
Let $D \subset \mathbb{H}^{3}$ be a least area surface and $a \in \mathring D \subset \mathbb{H}^{3}$. Since $\Theta (D, a, r)$ is increasing with $r$, we have that: 
\begin{center}
$\Theta (D,a) =\lim_{t \rightarrow 0} \Theta (D, a, t) \leq \Theta (D,a,r) =\frac{A(D \cap B(a,r))}{4\pi\sinh^{2}(\frac{r}{2})}$, 
\end{center}
for any $0 < r < d(a, \partial D)$.  By continuity of the area function, we can extend this up to $r = d(a, \partial D)$.

Now, being smoothly immersed implies that $\Theta(D,a) \geq 1$ for all $a \in \mathring S$.  By the above, we have that $A(D \cap B(a,r)) \geq 4\pi\sinh^{2}(\frac{r}{2})$, for any $0 < r \leq d(a, \partial D)$, as desired. 
\end{proof}

%%%%%%%%%%%%%%%%%%%%%%%%%%%%%%%%%%%%%%%%%%

\section{Least area surfaces and short geodesics in hyperbolic $3$-manifolds} 
\label{sec:LA_surfaces_and_geo}

First, let us set some notation.  Let $M$ be a hyperbolic $3$-manifold. The universal cover of $M$ is $\mathbb{H}^{3}$, and there exists a covering map $\rho: \mathbb{H}^{3} \rightarrow M$. Let $T_{r}(\gamma)$ denote an embedded tubular neighborhood of radius $r$ about a closed geodesic $\gamma \subset M$. $\gamma$ lifts to a geodesic, $\tilde{\gamma}$, in $\mathbb{H}^{3}$, and we will assume that the endpoints of $\tilde{\gamma}$ are $0$ and $\infty$. Let $T_{r}(\tilde{\gamma})$ be the tubular neighborhood of radius $r$ about $\tilde{\gamma}$ in $\mathbb{H}^{3}$.

Let $F$ be a surface in $M$ realized by the map $\varphi: S \rightarrow F$. Suppose $\gamma \cap F \neq \emptyset$, and say $p_{0} = \varphi(s_{0}) \in  \gamma \cap F \subset M$. Let $\tilde{S}$ be the universal cover of $S$, and denote by $\rho_{1}$ the covering map $\rho_{1}: \tilde{S} \rightarrow S$. Let $\tilde{s_{0}} \in \tilde{S}$ be a point with $\rho_{1}(\tilde{s_{0}}) = s_{0}$ and let $\tilde{\varphi}: \tilde{S} \rightarrow \mathbb{H}^{3}$ be a lift of $\varphi$ such that $\tilde{p_{0}} = \tilde{\varphi}(\tilde{s_{0}})$ is a point in $\tilde{\gamma}$. We have the following commutative diagram.   

\begin{center}
$\begin{CD}
(\tilde{S}, \tilde{s_{0}}) @> \tilde{\varphi} >> (\mathbb{H}^{3}, \tilde{p_{0}})\\
@VV\rho_{1}V @VV\rho V\\
(S, s_{0}) @>\varphi>> (M, p_{0})
\end{CD}$
\end{center}

The focus of the following subsections is to prove a number of propositions that can tell us when $\gamma$ can be isotoped disjoint from $F$ based on a variety of geometric and topological properties. Specifically, we will be interested in the tube radius of $\gamma$, the length of $\gamma$, and particular Dehn filling slopes. We will then use these conditions to show when the initial length spectrum can be preserved under mutation.  We will always be working with an \textit{almost least area surface} $F$ that is incompressible and $\partial$-incompressible in a hyperbolic $3$-manifold $M$.  The existence and embeddedness of such surfaces is provided by the following result of Ruberman. First, we define an almost least area surface.

\begin{defn}[Almost Least Area Surface in $M$]
\label{def:ALAS}
A properly and smoothly embedded surface $F$ in a Riemannian $3$-manifold $M$ is called \textit{almost least area} if $F$ is either a least area surface (as given in Definition \ref{defn:LA}), or is the boundary of an $\epsilon$-neighborhood of a one-sided embedded least area surface $F'$.  
\end{defn}

\textbf{Remark:} Theorems about almost least area surfaces hold for all $\epsilon$ sufficiently small. 
\newline

For the rest of Section \ref{sec:LA_surfaces_and_geo}, we will assume that any surface $F \subset M$ is a properly and smoothly embedded surface inside of a hyperbolic $3$-manifold $M$. 

\begin{thm} \cite[Theorem 1.6]{Ru}
\label{thm:LAsurfaces}
Let $F \subset M$ be a surface that is incompressible and $\partial$-incompressible.  Then $F$ can be properly isotoped to an almost least area surface.
\end{thm}

%In the next three subsections, we shall be using the following functions to describe when $\gamma \cap F = \emptyset$:

%%%%%%%%%%%%%%%%%%%%%%%%%%%%%%%%%%%%%%%%%%%%%%%%%%%

\subsection{Least area surfaces and the tube radius of $\gamma$}
\label{subsec:LA_surfaces}

The following proposition tells us that a closed geodesic $\gamma$ can be isotoped disjoint from an incompressible surface, if $\gamma$ has a sufficiently large embedded tubular radius. This fact can also be shown using \cite[Lemma 4.3]{FP2}.  However, here we provide additional geometric information about $\gamma \cap F$, when $F$ is in almost least area form. Recall that by a closed curve $n \cdot \gamma$, we mean a simple closed curve that is in the homotopy class of $[n \cdot \gamma] \in \pi_{1}(\partial T_{r}(\gamma))$.

\begin{prop}
\label{thm:LA_surface_disjoint}
Let $\gamma \subset M$ be a closed geodesic with embedded tubular radius $r$, and let $F$ be a surface in $M$ that is incompressible and $\partial$-incompressible.  Set $h(x) = 2\sinh^{-1}(\sqrt{\frac{x}{2}})$. Assume  $r> h( \left|\chi(F) \right| )$.  Then $\gamma$ can be isotoped disjoint from $F$. Furthermore, if $F$ is in almost least area form, then either $\gamma \cap F = \emptyset$ without any isotopy or $n \cdot \gamma$ is isotopic into $F$ for some $n \in \mathbb{N}$. In particular, if $\left|\chi(F) \right| \leq 2$, then our result holds whenever $r > 2 \ln(1 + \sqrt{2}) $.
\end{prop}

In order to prove this proposition, we will need the following lemma, which gives a lower bound on the area of a least area disk inside a tubular neighborhood of a geodesic.

\begin{lemma}
\label{lemma:LA_ surface_in_tube} 
Let $\gamma \subset M$ be a closed geodesic with embedded tubular neighborhood $T_{r}(\gamma)$.  Suppose $D_{r}$ is a least area disk in $M$ such that $\gamma \cap D_{r} \neq \emptyset$ and $ \partial D_{r} \subset \partial T_{r}(\gamma)$. Then $A(D_{r} \cap T_{r}(\gamma)) \geq 4\pi\sinh^{2}(\frac{r}{2})$.  
\end{lemma}

\begin{proof}
Since $\pi_{1}(D_{r})$ is trivial, $D_{r}$ lifts isometrically to a disk $\tilde{D_{r}} \subset T_{r}(\tilde{\gamma}) \subset \mathbb{H}^{3}$, with $\partial \tilde{D_{r}} \subset \partial T_{r}(\tilde{\gamma})$ and $\tilde{p_{0}} \in \tilde{D_{r}} \cap \tilde{\gamma}$. Since $D_{r}$ is least area and $D_{r}$ lifts isometrically to $\tilde{D_{r}}$, $\tilde{D_{r}}$ is a least area disk in $\mathbb{H}^{3}$ for the boundary curve $c = \partial \tilde{D_{r}}$. See figure \ref{Disk_in_Tube}.  By Corollary \ref{cor2}, $A(\tilde{D_{r}} \cap B(\tilde{p_{0}}, r)) \geq  4\pi\sinh^{2}(\frac{r}{2})$. Therefore,

\begin{center}
$A(D_{r}) = A(\tilde{D_{r}}) \geq 4\pi\sinh^{2}(\frac{r}{2})$,
\end{center}
as desired.
\end{proof}

\begin{figure}[ht]
\includegraphics[scale=0.65]{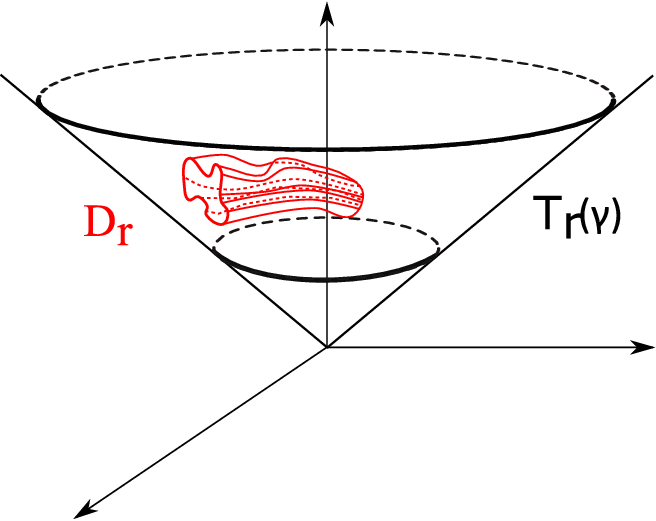}
\caption{The lift of a disk $D_{r}$ to $\mathbb{H}^{3}$ in the upper half-space model.}
\label{Disk_in_Tube}
\end{figure}

\begin{proof}[Proof of Proposition \ref{thm:LA_surface_disjoint}] 

Assume that $F$ has been isotoped to an (embedded) almost least area surface, as provided by Theorem \ref{thm:LAsurfaces}. Set $F_{r}= F \cap T_{r}(\gamma)$. We will always choose $r$ so that $F$ intersects $\partial(T_{r}(\gamma))$ transversely. By Sard's Theorem, this will hold for almost every $r$. Assume that $\gamma \cap F \neq \emptyset$.

\textbf{Claim:} $F_{r}$ is incompressible in $T_{r}(\gamma)$, and consequently, each component of $F_{r}$ is a disk or annulus.

Suppose that $F_{r}$ is compressible in $T_{r}(\gamma)$. Then there exists a disk $D' \subset T_{r}(\gamma)$ with $\partial D' \subset F_{r}$, but $\partial D'$ does not bound a disk in $F_{r}$. Since $F$ is incompressible in $M$, $\partial D'$ bounds a disk in $F$ which must lie at least partially outside of $T_{r}(\gamma)$. Call this disk $D$. Lift $D$ isometrically to a disk $\tilde{D} \subset \mathbb{H}^{3}$ with $\tilde{\partial D} \subset {T}_{r}(\tilde{\gamma})$. Now, $\tilde{D}$ can be homotoped to a disk (keeping $\tilde{\partial D}$ fixed) that lies on ${T}_{r}(\tilde{\gamma})$ via a nearest point projection map. 

We claim that this homotopy is area-decreasing. For this, we give $\mathbb{H}^{3}$ coordinates $(\rho, \theta, z)$, where $\rho \in \left( 0, \infty \right) $, $ \theta \in \left[ 0, 2\pi \right]$, and $z \in \mathbb{R}$. A point in $\mathbb{H}^{3}$ with coordinates $(\rho, \theta, z)$ has distance $\rho$ to the point on $\tilde{\gamma}$ at signed distance $z$ from $(0,0,1)$, and $\theta$ is the polar angle coordinate of its projections to the $(x,y)$-plane. A direct computation shows that
\begin{center}
	 $(\rho, \theta, z) = e^{z}(\tanh \rho \cos \theta, \tanh \rho \sin \theta, \sech \rho)$
\end{center}
pulls back the hyperbolic metric on the upper half-space model to the diagonal metric with respective diagonal entries $1$, $\sinh^{2} \rho$, and $\cosh^{2} \rho$. The nearest point projection to  ${T}_{r}(\tilde{\gamma})$ in these coordinates is given by  $(\rho, \theta, z) \rightarrow  (r, \theta, z)$, for $ \rho \geq r$. A direct computation shows that this projection reduces the area form of $\tilde{D}$ pointwise. Thus, projecting $\tilde{D}$ onto $\partial {T}_{r}(\tilde{\gamma})$ will give an area-decreasing homotopy. 

%Now suppose that the disk $\tilde{D}$ does not lie completely outside of ${T}_{r}(\tilde{\gamma})$. Then choose a sufficiently large radius $k$, with $k>r$, such that $T_{k}(\tilde{\gamma}) \cap \tilde{D}$ is a set of closed curves that bound disks lying completely outside of $T_{k}(\tilde{\gamma})$. We can then apply the argument from the previous paragraph to get an area-decreasing homotopy via projection of these disks onto $T_{k}(\tilde{\gamma})$. 

Projecting this homotopy down to $M$ yields an area-decreasing homotopy of $F$, which is a contradiction if $F$ is a least area surface. If $F$ is an $\epsilon$-neighborhood of a one-sided embedded least area surface $F'$, then we choose $\epsilon$ sufficiently small so that the strict area inequality we get from this homotopy still holds as an inequality. 

Thus, $F_{r}$ is incompressible in $T_{r}(\tilde{\gamma})$. The only incompressible surfaces with boundary that can be inside of $T_{r}(\tilde{\gamma})$ are essential disks and annuli.

We will now consider the two possibilities for the geometry of $\gamma \cap F$, when $F$ is in almost least area form.

\textbf{Case 1:} A component of $F_{r}$  is a disk that intersects $\gamma$.
\\
Say $D_{r}$ is a disk component of $F_{r}$ that intersects $\gamma$. If $F$ is in least area form, then we have the following area inequality:
\begin{center}
$2\pi \left| \chi(F) \right| \geq A(F) > A(F_{r}) \geq A(D_{r} \cap T_{r}(\gamma))  \geq 4\pi\sinh^{2}(\frac{r}{2})$.
\end{center}
The first inequality comes from the Gauss-Bonnet Theorem, combined with properties of minimal surfaces (see Futer--Purcell \cite[Lemma $3.7$]{FP2}). The last inequality comes from Lemma \ref{lemma:LA_ surface_in_tube}. Note that, we have a strict inequality for a least area surface, and by taking $\epsilon$ sufficiently small, the inequality $2\pi \left| \chi(F) \right| \geq 4\pi\sinh^{2}(\frac{r}{2})$ still holds if $F$ has been homotoped from a least area surface to an $\epsilon$-neighborhood of a one-sided least area surface. This gives us that
$\sqrt{\frac{\left| \chi(F) \right|}{2}} \geq \sinh({\frac{r}{2}})$.
Recall that $\sinh^{-1}(y) = \ln( y + \sqrt{y^{2}+1})$ and that $\sinh(x)$ is an increasing function.  Thus,

\begin{center}
$h(\left| \chi(F) \right|) = 2 \sinh^{-1}(\sqrt{\frac{\left| \chi(F) \right|}{2}}) = 2 \ln ( \sqrt{\frac{\left| \chi(F) \right|}{2}} + \sqrt{\frac{\left| \chi(F) \right|}{2} +1} ) \geq r.$
\end{center}

So, if $\gamma$ has a large enough embedded tubular radius, we will have a contradiction, specifically, if $r > h(\left| \chi(F) \right|)$.  In particular, if $\left|\chi(F) \right| \leq 2$, then $r > 2 \ln(1 + \sqrt{2}) =  h( \left|\chi(F) \right| )$ will provide the necessary area contradiction, and so, $\gamma \cap F = \emptyset$.

\textbf{Case 2:} Every component of $F_{r}$ that intersects $\gamma$ is an annulus.
\\
Suppose $A_{r}$ is an annulus component of $F_{r}$ that intersects $\gamma$.  In this case, the inclusion map $i: A_{r} \rightarrow T_{r}(\gamma)$, induces an injective homomorphism $i_{\ast}: \pi_{1} ( A_{r}) \hookrightarrow \pi_{1} ( T_{r}(\gamma) )$ with $[\alpha] \mapsto [n \cdot \gamma]$ for some $n \in \mathbb{N}$, where $[\alpha]$ is the homotopy class of the core of the annulus $A_{r}$. Now, $[\alpha]$ can be represented by a curve $\alpha$ on a component of $\partial A_{r}$, with $A_{r}$ providing the isotopy between the core and the boundary component. Since $\partial A_{r} \subset \partial T_{r}(\gamma)$, $\alpha$ is isotopic into the boundary torus $\partial T_{r}(\gamma)$, providing a satellite knot of the form $n \cdot \gamma$ on $\partial T_{r}(\gamma)$.

Finally, we show that our topological statement holds, that is, $\gamma$ can be isotoped disjoint from $F$ in both cases. Obviously, if $\gamma \cap F = \emptyset$, then no isotopy needs to even take place. So, suppose $n \cdot \gamma$ is isotopic into $F$.  The proof of case $2$ explains the topology of such a situation.  Specifically, the annuli $\left\{A_{r}^{i}\right\}_{i=1}^{n}$ are boundary parallel to $\partial T_{r}(\gamma)$, and so, could be isotoped disjoint from $\gamma$. If $F_{r}$ consists of multiple annuli that intersect $\gamma$, then we start by isotoping the outermost annuli to the boundary and proceed inward. Equivalently, we could keep $\left\{A_{r}^{i}\right\}_{i=1}^{n}$ fixed (since it is part of our least area surface $F$) and isotope $\gamma$ so that this closed curve is disjoint from $\left\{A_{r}^{i}\right\}_{i=1}^{n}$, and more generally, disjoint from $F$.  
\end{proof}

It is important to note that case $2$ of Proposition \ref{thm:LA_surface_disjoint} is certainly a possiblity and can be an obstruction to a useful lower bound estimate on $A(F)$. Techniques similar to the proof of Theorem \ref{thm:Monotonicity_of_MR} can be used to find a lower bound for $A(F \cap T_{r}(\gamma))$ when every component is an annulus, but the lower bound is of the form $C_{0} \cdot \ell(\gamma)  \cdot \sinh(r)$, where $C_{0} >0$ is a constant. It is possible to put a hyperbolic metric on a given surface $F$ so that a specific geodesic is arbitrarily short and contains an embedded collar of area $2 \ell(\gamma) \cdot \sinh(r)$. So, if $\ell(\gamma)$ is sufficiently short and $\gamma$ actually lies on $F$, then the quantity $C_{0} \cdot \ell(\gamma) \cdot \sinh(r)$ could be too small to be useful for our purposes.

%%%%%%%%%%%%%%%%%%%%%%%%%%%%%%%%%%%

\subsection{Least area surfaces and the length of $\gamma$}
\label{subsec:LA_surfaces_length}

Next, we will examine when $\gamma$ can be isotoped disjoint from $F$ based on the length of $\gamma$.  To do this, we will need to use the Collar Lemma, which essentially says that the shorter the length of a closed geodesic in a hyperbolic $3$-manifold, the larger the embedded tubular neighborhood of that geodesic.  The following qualitative version of the Collar Lemma comes from Meyerhoff \cite{Me}:

\begin{thm}[Collar Lemma]
\label{thm:Collar_lemma}
Let $\gamma \subset M$ be a closed geodesic in a hyperbolic $3$-manifold with (real) length $\ell(\gamma)$.  Suppose $\ell(\gamma) < \frac{\sqrt{3}}{4\pi} \left[\ln ( \sqrt{2} + 1)\right]^{2} \approx 0.107$.  Then there exists an embedded tubular neighborhood around $\gamma$ whose radius $r$ satisfies 
\begin{center}
$\sinh^{2}(r) = \frac{1}{2} \left(\frac{\sqrt{1-2k(\ell(\gamma))}}{k(\ell(\gamma))} - 1\right)$ where $k(x) = \cosh \left( \sqrt{ \frac{4 \pi x}{\sqrt{3}}} \right) - 1$.
\end{center}
\end{thm}

\begin{prop}
\label{thm:LA_surface_disjoint_corelength}
Let $\gamma \subset M$ be a closed geodesic, and let $F$ be a surface in $M$ that is incompressible and $\partial$-incompressible.  Set $g(x) = 2x^{2}+4x+1$. Assume $\frac{\sqrt{1-2k(\ell(\gamma))}}{k(\ell(\gamma))} > g( \left| \chi(F) \right|)$. Then $\gamma$ can be isotoped disjoint from $F$. Furthermore, if $F$ is in almost least area form, then either $\gamma \cap F = \emptyset$ without any isotopy or $n \cdot \gamma$ is isotopic into $F$ for some $n \in \mathbb{N}$. In particular, if $\left|\chi(F) \right| \leq 2$ our result holds whenever $\ell(\gamma) < 0.015$.
\end{prop}

\begin{proof}
We will use the Collar Lemma to show that if $\ell(\gamma)$ is sufficiently small, then the tube radius $r$ is sufficiently large.   Then Proposition \ref{thm:LA_surface_disjoint} will give us the desired result. So, we need to see when $r > h(\left| \chi(F) \right|) = 2 \sinh^{-1}(\sqrt{\frac{\left| \chi(F) \right|}{2}})$. Assume that $\ell(\gamma) < 0.107$, so the Collar Lemma applies.  Then we have $\sinh^{2}(r) = \frac{1}{2} \left(\frac{\sqrt{1-2k}}{k} - 1\right)$ where $k(\ell(\gamma)) = \cosh \left( \sqrt{ \frac{4 \pi \ell(\gamma)}{\sqrt{3}}} \right) - 1$. Now, $k(\ell(\gamma))$ is an increasing function on $0 < \ell(\gamma) < \infty$ with $k(\ell(\gamma)) \rightarrow 0$ as $\ell(\gamma) \rightarrow 0$, while $\frac{1}{2} \left(\frac{\sqrt{1-2k}}{k} - 1\right)$ is a decreasing function ($0 < k \leq \frac{1}{2}$), which heads to $\infty$ as $k \rightarrow 0$.  So, as $\ell(\gamma) \rightarrow 0$, $\sinh^{2}(r) = \frac{1}{2} \left(\frac{\sqrt{1-2k}}{k} - 1\right) \rightarrow \infty$.

Specifically, we need the following inequality to hold:

\begin{eqnarray*}
r = \sinh^{-1}(\sqrt{\frac{1}{2}(\frac{\sqrt{1-2k}}{k}-1})) & > & 2 \sinh^{-1}(\sqrt{\frac{\left| \chi(F) \right|}{2}}),\\
\frac{1}{2} (\frac{\sqrt{1-2k}}{k}-1) & > &  \sinh^{2}( 2 \sinh^{-1}(\sqrt{\frac{\left| \chi(F) \right|}{2}})), \\
\frac{\sqrt{1-2k}}{k} & > & 2 \sinh^{2}( 2 \sinh^{-1}(\sqrt{\frac{\left| \chi(F) \right|}{2}}))+1. 
\end{eqnarray*}

Note that,
\begin{eqnarray*}
 2 \sinh^{2}( 2 \sinh^{-1}(\sqrt{\frac{\left| \chi(F) \right|}{2}}))+1 & = & 2\sinh^{2}(\sinh^{-1}(2\sqrt{\frac{\left| \chi(F) \right|}{2}}\sqrt{\frac{\left| \chi(F) \right|}{2}+1}))+1 \\
 & = & 	2(2\sqrt{\frac{\left| \chi(F) \right|}{2}}\sqrt{\frac{\left| \chi(F) \right|}{2}+1})^{2}+1 \\
 & = & 2\left| \chi(F) \right|^{2} + 4\left| \chi(F) \right| +1 \\
 & = & g( \left| \chi(F) \right|).
\end{eqnarray*}
 
For the case when $\left|\chi(F) \right| \leq 2$, we just need to check when the inequality 
\begin{center}
$ \left(\frac{\sqrt{1-2k(\ell(\gamma))}}{k(\ell(\gamma))} \right) > g(2) = 17$
\end{center}
is satisfied. This occurs when $\ell(\gamma) < 0.015$, giving the desired result.
\end{proof}

%%%%%%%%%%%%%%%%%%%%%%%%%%%%%%%%%%%

\subsection{Least area surfaces and Dehn filling slopes}
\label{subsec:LA_surfaces_Dehn_filling}

Now, we would like to examine the geometry and topology of $\gamma \cap F$ based on certain Dehn filling slopes.  In order to do this, we need to go over some background on Dehn fillings.

Given a hyperbolic $3$-manifold $M$ with  a cusp corresponding to a torus boundary on $\partial M$, we choose a basis $\left\langle m,l \right\rangle$ for the fundamental group of the torus.  After this choice of basis, we can form the manifold $M\left(p,q\right)$ obtained by doing a $\left(p,q \right)$-Dehn surgery on the cusp, where $\left(p,q \right)$ is a coprime pair of integers.  A \emph{$\left(p,q \right)$-Dehn surgery} maps the boundary of the meridian disk to $s = pm+ql$. Similary, we can form the manifold $M\left((p_{1},q_{1}), \dots, (p_{k}, q_{k})\right)$ by performing a $\left(p_{i},q_{i}\right)$-Dehn surgery on the $i^{th}$ cusp of $M$, for each $i$, $ 1 \leq i \leq k$.

Thurston showed that $M\left((p_{1},q_{1}), \dots, (p_{k}, q_{k})\right)$ is in fact a hyperbolic $3$-manifold for all $\left((p_{1},q_{1}) \dots (p_{k}, q_{k})  \right)$ near $\left(\infty,\dots, \infty\right)$; see \cite{Th}. Following Thurston's work, many people developed techniques to more explicitly understand the change in geometry under Dehn surgery. The work of Hodgson and Kerckhoff \cite{HoKe2}, \cite{HoKe} shows that if the \textit{normalized lengths} of the slopes on which Dehn fillings are performed are sufficiently large, then it is possible to give explicit bounds on the geometry of the filled manifold. Their work will be helpful for us to determine when core geodesics (coming from Dehn filling) can be isotoped disjoint from incompressible surfaces inside of $M\left((p_{1},q_{1}), \dots, (p_{k}, q_{k})\right)$. We now define normalized length.

\begin{defn}[Normalized Length]
\label{defn:NL}
Given a Euclidean torus $T$, the \emph{normalized length of a slope $s = pm + ql$} is defined to be: 
\begin{center}
$\widehat{L}(s) = \widehat{L}((p,q))= \frac{\text{Length}((p,q))}{\sqrt{\text{Area}(T)}}$,
\end{center}
where Length($(p,q)$) is defined to be the length of a geodesic representative of $s$ on $T$. If we are considering multiple slopes, $\left\{s_{i}\right\}_{i=1}^{k}$, then define $\widehat{L}$ by the equation $\frac{1}{\widehat{L}^{2}} = \sum_{i=1}^{k} \frac{1}{\widehat{L}(s_{i})^{2}}$.
\end{defn}

Note that, normalized length is scale invariant and well-defined for cusps of $M$.

We now introduce some functions and terminology needed to understand certain results we will use from \cite{HoKe}. For the rest of this section, $M$ and $N$ will denote hyperbolic $3$-manifolds such that $M = N\left((p_{1},q_{1}), \dots, (p_{k}, q_{k})\right)$.  Each of these Dehn fillings produces a solid torus in $M$ whose core geodesic will be denoted by $\gamma_{i}$. We will use $r_{i}$ to denote the maximal embedded tube radus of $\gamma_{i}$. Section $5.1$ of \cite{HoKe} defines the  \textit{visual area} of the boundary of such an embedded tube and observes that it is equal to $\ell(\gamma_{i})\alpha_{i}$, where $\alpha_{i}$ is the cone angle around $\gamma_{i}$ (see above (25) on page 1068 there). Since $M$ is a manifold, its \textit{total visual area}, i.e., the sum of the visual areas of all tube boundaries, is $A = 2\pi \sum_{i=1}^{k} \ell(\gamma_{i})$.

The following two theorems come from \cite{HoKe}. The first relates the normalized lengths to the tube radii of the core geodesics resulting from Dehn filling, and the second relates these normalized lengths to the total visual area. The functions $f(z)$, $A(z)$, and $I(z)$ used in these theorems are given below. Also, $f(z)$ is formula $43$ on page $1080$ of \cite{HoKe}, and $A(z)$ is given on page $1080$ of \cite{HoKe} (though it is defined in terms of a function $H(z)$ given on page $1079$).

\begin{itemize}
	\item $f(z) = 3.3957(1-z) \exp (-\int_{1}^{z} F(w) dw)$, where $F(w) = \frac{-(1+4w+6w^{2}+w^{4})}{(w+1)(1+w^{2})^{2}}$,
	\item $A(z) = \frac{3.3957z(1-z^{2})}{1+z^{2}}$,
	\item $I(z) = \frac{(2\pi)^{2}}{f(z)}$.
\end{itemize}

\begin{thm}\cite{HoKe}
\label{thm:NLgivesradius}
Suppose that $M$ is obtained from $N$ by Dehn filling along slopes whose normalized lengths satisfy $\widehat{L}  > 7.5832$. If $\widehat{L}^{2} \geq I(z)$, then the tube radius $r_{i}$ of each $\gamma_{i}$ stays larger than $\rho = \tanh^{-1}(z)$.
\end{thm}

Theorem \ref{thm:NLgivesradius} is a slightly different version of Theorem $5.7$ from \cite{HoKe}. In \cite[Theorem 5.7]{HoKe}, the conclusion states that the tube radius of each $\gamma_{i}$ stays larger than a fixed radius $R_{0} = \tanh(\frac{1}{\sqrt{3}})$. In our version, the tube radius parameter is not fixed, but rather, a lower bound for it is given in terms of $\rho$. The two paragraphs preceding Theorem $5.7$ in \cite{HoKe} justify this change. Specifically, the bottom of page $1080$ and the top of page $1081$ state that the tube radius will be greater than or equal to $\rho$, provided that $\tanh(\rho)$ is greater than a particular minimum value: $\tanh(R_{0})$. Thus, to guarantee a larger tube radius, we must be able to choose larger values of $z$ (and hence larger values of $\rho$ too). This is accomplished by choosing  $\widehat{L}^{2} \geq I(z)$.

\begin{thm}\cite[Theorem 5.12]{HoKe}
\label{thm:NLgiveslength}
Suppose that $M$ is obtained from $N$ by Dehn filling along slopes whose normalized lengths satisfy $\widehat{L} > 7.5832$. Then the total visual area $A$ satisfies $A \leq A(z)$ where the variable $z$ is determined by $f(z) = \frac{(2\pi)^{2}}{\widehat{L}^{2}}$.
\end{thm}

The following proposition explicitly relates the normalized length of Dehn fillings to the geometry of the resulting core geodesics.

\begin{prop}
\label{thm:LA_surface_disjoint_conedef}
Suppose $M = N\left((p_{1},q_{1}), \dots, (p_{k}, q_{k})\right)$. Let $\left\{\gamma_{i}\right\}_{i=1}^{k} \subset M$ denote the set of closed geodesics which come from the cores of the solid tori obtained from Dehn filling cusps of $N$, and let $r_{i}$ denote the maximal embedded tube radius of $\gamma_{i}$. 
\begin{itemize}
\item If for each $i = 1, \dots, k$ we have $\widehat{L}((p_{i}, q_{i})) \geq 14.90\sqrt{k}$, then $r_{i} > 2 \ln(1 + \sqrt{2})$.
\item If for each $i = 1, \dots, k$ we have $\widehat{L}((p_{i}, q_{i})) \geq 20.76\sqrt{k}$, then $\ell(\gamma_{i}) < 0.015$.
\end{itemize}
\end{prop}

\begin{proof}
For the first bullet, we use Theorem \ref{thm:NLgivesradius} to guarantee each tube radius $r_{i}$ is sufficiently large by making each normalized length $\widehat{L}((p_{i}, q_{i}))$ sufficiently large. Specifically, we require $\widehat{L}^{2} \geq I(z)$, for $z = \tanh(2 \ln(1 + \sqrt{2}))$ to guarantee that the tube radius of each $\gamma_{i}$ is at least $2(\ln(1+\sqrt{2}))$. Since for each $i = 1, \dots, k$ we have $\widehat{L}((p_{i}, q_{i})) \geq 14.90\sqrt{k}$, it follows that 
\begin{center}
	$\frac{1}{\widehat{L}^{2}} = \sum_{i=1}^{k} \frac{1}{\widehat{L}(p_{i}, q_{i})^{2}}  \leq (k) (\frac{1}{14.90 \sqrt{k}})^{2} = \frac{1}{222.01}$.	
\end{center} 
Thus, $\widehat{L}^{2} \geq 222.01$. Doing the necessary algebra reveals that $222.01 \geq I(z)$ when 
  $z =  \tanh(2 \ln(1 + \sqrt{2}))$, giving the desired result. 

Now we consider the second bullet. For the filled manifold $M$, we have that the total visual area $A = 2\pi \sum_{i=1}^{k} \ell(\gamma_{i})$. In our case, we want each $\ell(\gamma_{i}) < 0.015$, which will certainly be true if $\sum_{i=1}^{k} \ell(\gamma_{i}) < 0.015$. Thus, if $A \leq 2\pi(0.015)$ then each geodesic $\gamma_{i}$ will be sufficiently short. By Theorem \ref{thm:NLgiveslength}, we know that $A \leq A(z) = \frac{3.3957z(1-z^{2})}{1+z^{2}}$, where the variable $z$ is determined by the equation $f(z) = \frac{(2\pi)^{2}}{\widehat{L}}$. Thus, we need to choose our $\widehat{L}((p_{i},q_{i}))$ sufficiently large so that $z$ satisfies $A(z) \leq 2\pi(0.015)$. Doing some algebra yields the following.
\begin{center}
$\widehat{L} = \sqrt{\frac{(2\pi)^{2}}{f(z)}} = \sqrt{\frac{(2\pi)^{2} \exp (\int_{1}^{z} F(w) dw)}{3.3957(1-z)}}$. 
\end{center}
Choosing each $\widehat{L}((p_{i}, q_{i})) \geq 20.76\sqrt{k}$ results in $A(z) \leq 2\pi(0.015)$, as needed. 
\end{proof}

Either of these conditions will guarantee that any such core geodesics $\gamma_{i}$ can be isotoped disjoint from an incompressible surface $F$ with $\left|\chi(F) \right| \leq 2$. This comes from combining Proposition \ref{thm:LA_surface_disjoint_conedef} with Proposition \ref{thm:LA_surface_disjoint} in the first case and Proposition \ref{thm:LA_surface_disjoint_corelength} in the second case, respectively. However, while the lower bound on normalized length is smaller for the first bullet, in certain applications we will actually want to guarantee that not only our geodesics can be isotoped disjoint from $F$, but also, these geodesics are sufficiently short. This is why we include the second condition. These results are summarized in Corollary \ref{cor:disjointgeo} in the next section.

%%%%%%%%%%%%%%%%%%%%%%%%%%%%%%%%%%%%%%%%%%%%%%%%%%%%%%%%%%%%%%%%%

\subsection{Summary of conditions}
\label{subsec:summary}

We now summarize the conditions under which $\gamma$ can be isotoped disjoint from $F$. This will be used in the proof of Theorem \ref{thm:mutationrep} and its corollaries. 

\begin{thm}
\label{thm:gammasep}
Let $M$ be a hyperbolic manifold with $F \subset M$ a surface that is incompressible and $\partial$-incompressible. Let $\gamma \subset M$ be a closed geodesic with embedded tubular radius $r$. Assume
\begin{enumerate}
\item $r> h( \left|\chi(F) \right| )$, or
\item $\frac{\sqrt{1-2k(\ell(\gamma))}}{k(\ell(\gamma))} > g( \left| \chi(F) \right|)$.
\end{enumerate}

Then $\gamma$ can be isotoped disjoint from $F$. Furthermore, if $F$ is in almost least area form, then either $\gamma \cap F = \emptyset$ without any isotopy or $n \cdot \gamma$ is isotopic into $F$ for some $n \in \mathbb{N}$.
\end{thm}

\begin{proof}
Combine Proposition \ref{thm:LA_surface_disjoint} and Proposition \ref{thm:LA_surface_disjoint_corelength}. 
\end{proof}

Plugging in $\left|\chi(F) \right| \leq 2$ gives the following immediate corollary.

\begin{cor}
\label{cor:disjointexplicit}
Let $M$ be a hyperbolic manifold with $F \subset M$ a surface that is incompressible and $\partial$-incompressible with $\left|\chi(F) \right| \leq 2$. Let $\gamma \subset M$ be a closed geodesic with embedded tubular radius $r$. Assume
\begin{enumerate}
\item $r> 2 \ln (1 + \sqrt{2})$, or
\item $\ell(\gamma) < 0.015$.
\end{enumerate}

Then $\gamma$ can be isotoped disjoint from $F$. Furthermore, if $F$ is in almost least area form, then either $\gamma \cap F = \emptyset$ without any isotopy or $n \cdot \gamma$ is isotopic into $F$ for some $n \in \mathbb{N}$.
\end{cor}

For our applications, we will mainly be concerned with closed geodesics that are the core geodesics coming from Dehn fillings and surfaces $F_{i}$ with $\left|\chi(F_{i}) \right| \leq 2$. Thus, the following corollary will be useful, which comes from combining Corollary \ref{cor:disjointexplicit} with Proposition \ref{thm:LA_surface_disjoint_conedef}.

\begin{cor}
\label{cor:disjointgeo}
Suppose $M = N\left((p_{1},q_{1}), \dots, (p_{k}, q_{k})\right)$ and $F \subset M$ a surface that is incompressible and $\partial$-incompressible with $\left|\chi(F) \right| \leq 2$. Let $\left\{\gamma_{i}\right\}_{i=1}^{k} \subset M$ denote the core geodesics coming from Dehn filling cusps of N, each with embedded tube radius $r_{i}$.  

\begin{enumerate}
\item If for each $i = 1, \dots, k$ we have that $\widehat{L}((p_{i}, q_{i})) \geq 14.90\sqrt{k}$, then each $\gamma_{i}$ can be isotoped disjoint from $F$ and each $r_{i} > 2 \ln(1 + \sqrt{2})$.
\item If for each $i = 1, \dots, k$ we have that $\widehat{L}((p_{i}, q_{i})) \geq 20.76\sqrt{k}$, then in addition each $\ell(\gamma_{i}) < 0.015$.
\end{enumerate}

Furthermore, if $F$ is in almost least area form, then either $\gamma_{i} \cap F = \emptyset$ without any isotopy or $n \cdot \gamma_{i}$ is isotopic into $F$ for some $n \in \mathbb{N}$.
\end{cor}

Combining the results from this section gives a proof of Theorem \ref{thm:main} from the introduction. 

\begin{proof}[Proof of Theorem \ref{thm:main}] 
Corollary \ref{cor:disjointexplicit} takes care of the first two cases of Theorem \ref{thm:main}, while Corollary \ref{cor:disjointgeo} takes care of the third case by considering Dehn filling a single cusp. 
\end{proof}

%%%%%%%%%%%%%%%%%%%%%%%%%%%%%%%%%%%%%%%%%%%%%

\section{Hyperelliptic surfaces and mutations that preserve geodesics}
\label{subsec:symmsurf_and_mut}

In this section, we will prove that mutating along hyperelliptic surfaces inside hyperbolic $3$-manifolds preserves the initial (complex) length spectrum. In what follows, let $S_{g,n}$ denote a surface of genus $g$ and $n$ boundary components.

Recall that a hyperelliptic surface $S$ is a surface that admits at least one non-trivial involution $\mu$ of $S$ so that $\mu$ fixes every isotopy class of curves in $S$.  Note that, the surfaces $S_{2,0}$, $S_{1,2}$, $S_{1,1}$, $S_{0,3}$, and $S_{0,4}$ are always hyperelliptic, regardless of their hyperbolic structures.  Also, these are all surfaces with Euler characteristic $-1$ or $-2$.  For our constructions in Section \ref{sec:RT_and_PK}, we will examine $4$-punctured spheres that arise in hyperbolic knot complements.  An $S_{0,4}$ in a knot complement is called a \emph{Conway sphere}.

\begin{figure}[ht]
\includegraphics[scale=0.50]{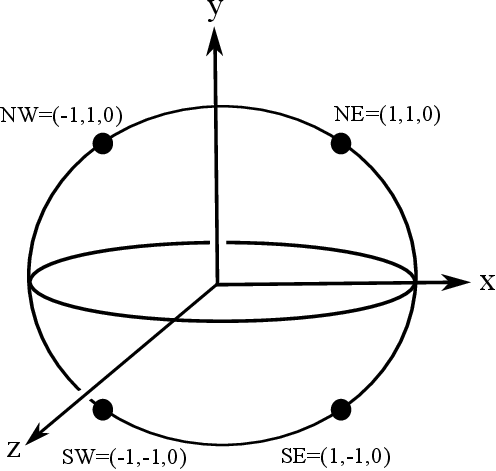}
\caption{A standard Conway sphere.}
\label{Conwaysphere}
\end{figure}

For a Conway sphere there are three hyperelliptic (orientation preserving) involutions, given by $180^{\circ}$ rotations about the $x$-axis, $y$-axis, and $z$-axis, respectively, as shown in figure \ref{Conwaysphere}. 

\begin{defn}[Mutation]
\label{def:Mutation}
A \emph{mutation} along a hyperelliptic surface $S$ in a $3$-manifold $M$ is the process of cutting $M$ along $S$ and then regluing by one of the nontrivial involutions of $S$ to obtain the $3$-manifold $M^{\mu}$. If $K$ is a knot in $\mathbb{S}^{3}$ with a Conway sphere $S$, then cutting $(S^{3}, K)$ along $(S, S \cap K)$ and regluing by a mutation, $\mu$, yields a knot $K^{\mu} \subset S^{3}$. 
\end{defn}

Corollary \ref{cor:disjointexplicit} will help us determine a lower bound on the number of geodesic lengths preserved under mutation. To do this, we first need to see how representations of $\pi_{1}(M$) and $\pi_{1}(M^{\mu})$ are related as amalgamated products and HNN-extensions along representations of $\pi_{1}(F)$. In fact, Kuessner in \cite{Ku} gives a different proof of Ruberman's result about mutations and volume that uses these decompositions of representations of $\pi_{1}(M$) and $\pi_{1}(M^{\mu})$ along with the Maskit combination theorem and homological arguments.

The following theorem due to Ruberman characterizes an essential feature of a hyperelliptic surface $(F, \mu)$.  

\begin{thm}\cite[Theorem $2.2$]{Ru}
\label{thm:mutationrep}
Let $(F, \mu)$ be a hyperelliptic surface, and let $\rho_{F}: \pi_{1}(F) \rightarrow \text{PSL}(2, \mathbb{C})$ be a discrete and faithful representation taking cusps of $F$ to parabolics. Then there exists  $\beta \in \text{PSL}(2, \mathbb{C})$ such that $\rho_{F}  \mu_{\ast} = \beta  \rho_{F} \beta^{-1}$. 
\end{thm}

Geometrically, this means that a hyperelliptic involution acts as a rigid motion of a fundamental domain for $\rho_{F}(\pi_{1}(F))$ in $\mathbb{H}^{3}$. 

In what follows, suppose that $M = \mathbb{H}^{3} / \Gamma$ where $\Gamma$ is the Kleinian group corresponding to the representation $\rho: \pi_{1}(M) \rightarrow \text{PSL}(2, \mathbb{C})$. In addition, assume that $(F, \mu)$ is a hyperelliptic surface inside of $M$, and mutation along $F$ produces $M^{\mu}$. If $F$ is separating in $M$, then assume cutting along $F$ decomposes $M$ into two pieces, $M_{a}$ and $M_{b}$. If $F$ is non-separating, then assume cutting along $F$ decomposes $M$ into $N$ where $\partial N = F_{1} \cup F_{2}$. Here, $F_{1}$ and $F_{2}$ are copies of $F$ and $M$ is the quotient of $N$ under some homeomorphism $\psi: F_{1} \rightarrow F_{2}$. Also, assume that $\Gamma_{a}$, $\Gamma_{b}$, $\Gamma_{F}$, and $\Gamma_{N}$ are Kleinian subgroups of $\Gamma$ that are isomorphic to $\pi_{1}(M_{a})$, $\pi_{1}(M_{b})$, $\pi_{1}(F)$, and $\pi_{1}(N)$, respectively, with these isomorphisms coming from restricting $\rho:\pi_{1}(M) \rightarrow \text{PSL}(2, \mathbb{C})$.  

 The previous paragraph tells us that $\Gamma = \left\langle \Gamma_{a}, \Gamma_{b} \right\rangle  \cong \Gamma_{a} \ast_{\Gamma_{F}} \Gamma_{b}$ when $F$ is separating and $\Gamma = \left\langle \Gamma_{N} , \gamma \right\rangle \cong \Gamma_{N} \ast_{\gamma}$ where $\gamma g \gamma^{-1} = \psi_{\ast}(g)$ for $g$ in the subgroup $\Gamma_{1}$ of $\Gamma_{N}$, when $F$ is non-separating. The following lemma shows that we also get a decomposition of $\Gamma^{\mu}$ in terms of $\Gamma_{a}$ and $\Gamma_{b}$. A similar lemma is given by Kuessner in \cite[Proposition 3.1]{Ku}.
 
 In the following lemma and theorem, we use $=$ to denote equality of Kleinian groups and $\cong$ to denote an abstract group isomorphism.

\begin{lemma}
\label{lemma:Fgroups}
Let $F \subset M$ be a properly embedded surface that is incompressible, \\ $\partial$-incompressible, and admits a hyperelliptic involution $\mu$. If $F$ is separating, then there exists $\beta \in \text{PSL}(2, \mathbb{C})$ such that 

\begin{center}
$\Gamma = \left\langle \Gamma_{a}, \Gamma_{b} \right\rangle  \cong \Gamma_{a} \ast_{\Gamma_{F}} \Gamma_{b}$ and $\Gamma^{\mu} = \left\langle \Gamma_{a}, \beta \Gamma_{b} \beta^{-1} \right\rangle  \cong \Gamma_{a} \ast_{\Gamma_{F}} \beta \Gamma_{b} \beta^{-1}$.
\end{center}
If $F$ is non-separating, then there exists $\beta \in \text{PSL}(2, \mathbb{C})$ such that
\begin{center}
$\Gamma = \left\langle \Gamma_{N} , \alpha \right\rangle \cong \Gamma_{N} \ast_{\alpha}$ and $\Gamma^{\mu} = \left\langle \Gamma_{N} , \alpha \beta \right\rangle \cong \Gamma_{N} \ast_{\alpha \beta}$,
\end{center}
where $\alpha g \alpha^{-1} = \psi_{\ast}(g)$ for $g$ in the subgroup $\Gamma_{1}$ of $\Gamma_{N}$ uniformizing $\pi_{1}(F)$ and $\beta$ normalizes $\Gamma_{1}$ with $\beta g \beta^{-1} = \mu_{\ast}(g)$.
\begin{flushleft}
In both cases, $\Gamma^{\mu}$ is discrete and $M^{\mu}$ is homeomorphic to $\mathbb{H}^{3} / \Gamma^{\mu}$.
\end{flushleft}

\end{lemma}

\textbf{Remark: }In Kuessner's version of this statement, he assumes that the surface $F$ is not a virtual fiber. However, after the proof of \cite[Proposition 3.1]{Ku}, Kuessner suggests a slight variation of his proof that removes this requirement. Here, we make no such requirement of $F$ and prove the more general case by following Kuessner's suggestion to utilize the least area surface machinery that Ruberman develops in \cite{Ru}.

\begin{proof}
Here, we give a proof of the case when $F$ is separating. The non-separating case is proved similarly, and we give a brief outline of this case at the end of this proof. Since $F$ is incompressible in M, $F$ is also incompressible in $M_{a}$ and $M_{b}$. Thus, the inclusion maps $i: F \rightarrow M_{a}$ and $j: F \rightarrow M_{b}$ induce monomorphisms $i_{\ast}: \pi_{1}(F) \rightarrow \pi_{1}(M_{a})$ and $j_{\ast}: \pi_{1}(F) \rightarrow \pi_{1}(M_{b})$, respectively. Let $\rho_{a}$ denote the restriction of $\rho$ to $\pi_{1}(M_{a})$, and similarly, let $\rho_{F}$ denote the restriction of $\rho$ to $\pi_{1}(F)$. Then the map $f_{1}: \Gamma_{F} \rightarrow \Gamma_{a}$ defined by $f_{1} = \rho_{a} i_{\ast} \rho_{F}^{-1}$ is a well-defined monomorphism. Similarly, we have a monomorphism $f_{2}: \Gamma_{F} \rightarrow \Gamma_{b}$, defined by $f_{2} = \rho_{b} j_{\ast} \rho_{F}^{-1}$, where $\rho_{b}$ denotes the restriction of $\rho$ to $\pi_{1}(M_{b})$. This tells us that $\Gamma \cong \Gamma_{a} \ast_{\Gamma_{F}} \Gamma_{b} \cong (\Gamma_{a} \ast \Gamma_{b}) / N$, where $N$ is the normal subgroup of $\Gamma_{a} \ast \Gamma_{b}$ generated by elements of the form $f_{1}(h)f_{2}(h)^{-1}$, for all $h \in \Gamma_{F}$. 

Now, $M^{\mu}$ is also constructed by cutting $M$ along $F$, and then gluing the pieces $M_{a}$ and $M_{b}$ back together along $F$.  However, we now rotate one of these pieces, say $M_{b}$, by the hyperelliptic involution $\mu$ before gluing it back to $M_{a}$ along $F$. Theorem \ref{thm:mutationrep} provides the existence of some $\beta \in \text{PSL}(2, \mathbb{C})$ such that $\rho_{F}  \mu_{\ast} = \beta  \rho_{F} \beta^{-1}$. Let $f_{\beta}: \Gamma_{b} \xrightarrow{\sim} \beta\Gamma_{b}\beta^{-1}$ be the map that conjugates by $\beta$. This gives us a well-defined monomorphism $f_{3}: \Gamma_{F} \rightarrow \beta \Gamma_{b} \beta^{-1}$ defined by $f_{3} = f_{\beta} f_{2}$. 

First, we will show that $\Gamma^{\mu}  \cong \Gamma_{a} \ast_{\Gamma_{F}} \beta \Gamma_{b} \beta^{-1} \cong (\Gamma_{a} \ast \beta \Gamma_{b} \beta^{-1}) / K$, where $K$ is the subgroup generated by elements of the form $f_{1}(h)f_{3}(h)^{-1}$ for all $h \in \Gamma_{F}$.  This group isomorphism will follow from the Maskit combination theorem \cite[VII.A.10]{Mas}. Assume that $F$ is isotopic to its least-area representative; the case where $F$ double covers a least-area representative is left to the reader. Ruberman's Theorem \ref{thm:mutationrep} implies that the element $\beta$ such that $\rho_{F}  \mu_{\ast} = \beta  \rho_{F} \beta^{-1}$ induces an isometric involution $\tilde{\tau}$ of the cover $M_{F} \rightarrow M$ corresponding to $\pi_{1}(F)$. In the proof of \cite[Theorem 1.3]{Ru}, Ruberman shows that a least-area representative of $F$ lifts to an embedding $\hat{F}$ in $M_{F}$. Furthermore, $\hat{F}$ is invariant under $\hat{\tau}$, and so, the preimage $\tilde{F}$ of $\hat{F}$ in $\mathbb{H}^{3}$ is $\beta$-invariant. Since $\tilde{F}$ is a properly embedded plane in $\mathbb{H}^{3}$, we have that $\mathbb{H}^{3} \setminus \tilde{F}$ decomposes into two (non-empty) $3$-balls, $B_{a}$ and $B_{b}$.   

We claim that $B_{a}$ and $B_{b}$ comprise a proper interactive pair of sets (in the sense of \cite[VII.A]{Mas}) for $\Gamma_{a}$ and  $\beta \Gamma_{b} \beta^{-1}$. Here, we can follow the same argument as Kuessner in \cite[Proposition 3.1]{Ku}, but replace the subsets $B_{1}$ and $B_{2}$ of $\partial_{\infty} \mathbb{H}^{3}$ with $B_{a}$ and $B_{b}$. The Maskit combination theorem then implies that $\Gamma^{\mu} = \left\langle \Gamma_{a}, \beta \Gamma_{b} \beta^{-1} \right\rangle  \cong \Gamma_{a} \ast_{\Gamma_{F}} \beta \Gamma_{b} \beta^{-1}$. The fact that $\Gamma^{\mu}$ is discrete follows from the argument in \cite[VII.C.4]{Mas}.     

Finally, we claim that $M^{\mu}$ is homeomorphic to $\mathbb{H}^{3} / \Gamma^{\mu}$. By applying van Kampen's Theorem, we have that $\pi_{1}M^{\mu}\cong \pi_{1}M_{a} \ast_{\pi_{1}F} \pi_{1}M_{b}$, where the respective inclusions of $\pi_{1}F$ are given by $i_{\ast}$ and $j_{\ast}\mu_{\ast}$. This gives an isomorphism $\rho^{\mu}: \pi_{1}M^{\mu} \rightarrow \Gamma^{\mu}$ defined on $\pi_{1}M_{a}$ by $\rho$ and on $\pi_{1}M_{b}$ by $\beta \rho \beta^{-1}$, as desired. 

To prove the non-separating case, we would use the Maskit combination theorem for HNN-extensions along with the same least-area surface argument used in the separating case to obtain the desired group isomorphism $\Gamma^{\mu} = \left\langle \Gamma_{N} , \alpha \beta \right\rangle \cong \Gamma_{N} \ast_{\alpha \beta}$ and discreteness of $\Gamma^{\mu}$. Again, $M^{\mu}$ will be homeomorphic to $\mathbb{H}^{3} / \Gamma^{\mu}$ by an application of van Kampen's Theorem.   
\end{proof}

By combining the previous lemma with Corollary \ref{cor:disjointexplicit} and Corollary \ref{cor:disjointgeo}, we can now give a number of scenarios for which mutation preserves a portion of the (complex) length spectrum. 

In what follows, let $G_{L}(M)$ denote the geodesics in $M$ that make up the initial length spectrum up to a cut off length of $L$, that is, 
\begin{center}
	$G_{L}(M) = \left\{ \gamma \subset M : \gamma \hspace{0.05in} \text{is a closed geodesic and} \hspace{0.05in} \ell(\gamma) < L \right\}$.
\end{center}

\begin{prop}
	\label{prop:homotopegeo}
Let $F \subset M$ be a surface that is incompressible, $\partial$-incompressible, and admits a hyperelliptic involution $\mu$.  Let $\gamma \subset M$ be a closed geodesic. If $\gamma$ can be homotoped disjoint from $F$, then there exists a geodesic $\gamma^{\mu} \subset M^{\mu}$ such that $\ell_{\mathbb{C}}(\gamma) = \ell_{\mathbb{C}}(\gamma^{\mu})$. Furthermore, suppose that every geodesic shorter than length $L$ can be homotoped disjoint from $F$ (in both $M$ and $M^{\mu}$). Then there is a bijection between the complex length spectra of $M$ and $M^\mu$ up to length $L$.
\end{prop}

\begin{proof}
Suppose $M$ and $M^{\mu}$ are hyperbolic $3$-manifolds that differ by mutation along $(F, \mu)$. Let $\gamma \subset M$ be any closed geodesic that can be homotoped disjoint from $F$. Assume we have performed this homotopy. In what follows, we abuse notation and let $\gamma$ refer to multiple representatives from the homotopy class $[\gamma] \in \pi_{1}(M)$, and not just the geodesic representative. Similarly for $[\gamma^{\mu}] \in \pi_{1}(M^{\mu})$.

First, suppose that $F$ separates $M$.  By Lemma \ref{lemma:Fgroups}, we have that $\Gamma = \left\langle \Gamma_{a}, \Gamma_{b} \right\rangle$ and $\Gamma^{\mu} = \left\langle \Gamma_{a}, \beta \Gamma_{b} \beta^{-1} \right\rangle$, for some $\beta \in \text{PSL}(2, \mathbb{C})$. Since $\gamma \subset M$ has been homotoped disjoint from $F$, $\gamma \in M_{a}$, or $\gamma \in  M_{b}$. Without loss of generality, assume $[\gamma] \in \pi_{1}(M_{a})$, i.e., $\gamma$ now lies in $M_{a}$. $[\gamma] \in \pi_{1}(M)$ has a unique (complex) length associated to it, $\ell_{\mathbb{C}}(\gamma)$, coming from the representation $\rho: \pi_{1}(M) \xrightarrow{\sim} \Gamma = \left\langle \Gamma_{a}, \Gamma_{b} \right\rangle \subset \text{PSL}(2, \mathbb{C})$.  This (complex) length is determined by the trace of its representation.  Specifically, $\cosh(\frac{\ell_{\mathbb{C}}(\gamma)}{2}) = \pm \frac{tr(\gamma)}{2}$, where $tr(\gamma)$ denotes the trace of the representation of $\gamma$.  Since we have homotoped $\gamma$ disjoint from $F$, mutating along $F$ to obtain $M^{\mu}$ will produce a corresponding homotopy class $[\gamma^{\mu}] \in \pi_{1}(M^{\mu})$. Similarly, $[\gamma^{\mu}] \in \pi_{1}(M^{\mu})$ also has a unique (complex) length associated to it, coming from $\rho_{\mu}: \pi_{1}(M^{\mu}) \xrightarrow{\sim} \Gamma^{\mu} = \left\langle \Gamma_{a}, \beta \Gamma_{b} \beta^{-1} \right\rangle \subset \text{PSL}(2, \mathbb{C})$. Thus, $[\gamma]$ and $[\gamma^{\mu}]$ have the same representation in $\text{PSL}(2, \mathbb{C})$ since $\rho$ and $\rho_{\mu}$ agree on $\pi_{1}(M_{a})$. So, the same complex length is associated to $\gamma$ and $\gamma^{\mu}$, as desired.  Note that, if $\gamma$ was homotoped into $M_{b}$ instead, then the representations of $\gamma$ and $\gamma^{\mu}$ into $\text{PSL}(2, \mathbb{C})$ would be conjugate to one another. Since trace is preserved by conjugation, the corresponding complex length will still be preserved too.

If $F$ is non-separating in $M$, then Lemma \ref{lemma:Fgroups} gives us that $\Gamma = \left\langle \Gamma_{N} , \alpha \right\rangle$  and $\Gamma^{\mu} = \left\langle \Gamma_{N} , \alpha \beta \right\rangle$.  Since $\gamma$ has been homotoped disjoint from $F$, we once again have that $[\gamma] \in \pi_{1}(N) \subset \pi_{1}(M)$ and $[\gamma^{\mu}] \in \pi_{1}(N) \subset \pi_{1}(M^{\mu})$ have the same representation in $\text{PSL}(2, \mathbb{C})$ (up to conjugation), and so, the same complex length associated to them.

Now, suppose that every geodesic shorter than length $L$ can be homotoped disjoint from $F$, and say this set of geodesics is $G_{L}(M) = \left\lbrace \gamma_{i}\right\rbrace_{i=1}^{n}$. By the first part of this proof, each $\gamma_{i}$ will have a mutant partner $\gamma_{i}^{\mu}$ in $M^{\mu}$ with the same complex length. We need to show that there exists a bijective correspondence between $G_{L}(M)$ and $G_{L}(M^{\mu})$. Let $f : G_{L}(M) \rightarrow G_{L}(M^{\mu})$ be the function defined by $f(\gamma_{i}) = \gamma_{i}^{\mu}$, for each $\gamma_{i} \in G_{L}(M)$.  This map is obviously one-to-one: if $\gamma_{i}^{\mu} = f(\gamma_{i}) = f(\gamma_{j}) = \gamma_{j}^{\mu}$, then mutating $M^{\mu}$ along $(F, \mu)$ to obtain $M$ implies $\gamma_{i} = \gamma_{j}$.  Now, suppose $f$ is not onto, and so, there exists some $\gamma^{\mu} \in G_{L}(M^{\mu})$ such that $\gamma^{\mu} \notin \left\{\gamma_{i}^{\mu}\right\}_{i=1}^{n}$. Mutate $M^{\mu}$ by ($F, \mu)$ to obtain $M$.  Since $\gamma^{\mu} \in G_{L}(M^{\mu})$, $\ell(\gamma^{\mu}) < L$, which implies that $\gamma^{\mu}$ can be homotoped disjoint from $F$. The first part of this proof implies that there is a corresponding $\gamma \in M$ with the same complex length as $\gamma^{\mu}$.  However, then $\ell(\gamma) < L$, i.e. $\gamma \in G_{L}(M)$, which is a contradiction.  Thus, $f$ gives a bijective correspondence between $G_{L}(M)$ and $G_{L}(M^{\mu})$, as desired.  
\end{proof}

The following corollary quickly follows from Proposition \ref{prop:homotopegeo} and Corollary \ref{cor:disjointexplicit}.

\begin{cor}
\label{cor:syspreserved2}
Let $F \subset M$ be a surface that is incompressible, $\partial$-incompressible, and admits a hyperelliptic involution $\mu$.  Then for any $L <0.015$, $G_{L}(M)$ is in bijective correspondence with $G_{L}(M^{\mu})$. In particular, if $M$ has $n$ geodesics shorter than $L$, then $M$ and $M^{\mu}$ have at least the same $n$ initial values of their respective complex length spectra.  
\end{cor}

\begin{proof}
	Suppose there are $n$ geodesics shorter than $L$, and set $\left\{\gamma_{i}\right\}_{i=1}^{n} = G_{L}(M)$.  By Corollary \ref{cor:disjointexplicit}, we can isotope (and so homotope) any such $\gamma_{i}$ disjoint from $F$. Proposition \ref{prop:homotopegeo} then implies that for each $\gamma_{i}$, there exists a corresponding closed geodesic $\gamma_{i}^{\mu}$ in $M^{\mu}$, such that $\ell_{\mathbb{C}}(\gamma_{i}) = \ell_{\mathbb{C}}(\gamma_{i}^{\mu})$. In addition, Proposition \ref{prop:homotopegeo} guarantees that we have the desired bijective correspondence between $G_{L}(M)$ and $G_{L}(M^{\mu})$.
\end{proof}

\textbf{Remark:} The above corollary uses the length condition from Corollary \ref{cor:disjointexplicit} to determine when $M$ and its mutant $M^{\mu}$ have the same initial length spectra.  We also get corollaries (highlighted below), based upon the tube radius condition and the normalized length condition.  However, with the tube radius condition, we can not guarantee that these common geodesic lengths are the shortest ones in the length spectra of $M$ and $M^{\mu}$, since there can exist geodesics with a very large embedded tube radius that are not very short. Thus, we can only say that that a portion of these length spectra are the same, not necessarily the initial length spectra. Fortunately, for the normalized length condition, we can still get a corollary that determines when $M$ and $M^{\mu}$ have the same initial length spectra, by using the second condition in Corollary \ref{cor:disjointgeo}. 

\begin{cor}
\label{cor:syspreservedr}
Let $F \subset M$ be a surface that is incompressible, $\partial$-incompressible, and admits a hyperelliptic involution $\mu$. Suppose that $M$ has exactly $n$ geodesics with embedded tubular radius larger than some constant $R > 2 \ln(1 + \sqrt{2})$. Then $M$ and $M^{\mu}$ have at least $n$ common values in their respective (complex) length spectra.  
\end{cor}

\begin{cor}
\label{cor:syspreservednl}
Let $F \subset M$ be a surface that is incompressible, $\partial$-incompressible, and admits a hyperelliptic involution $\mu$. Suppose that $M$ has exactly $n$ geodesics that are the core geodesics coming from Dehn filling a hyperbolic $3$-manifold $N$. Let $\widehat{L}(s_{i})$ denote the normalized slope length of the $i^{th}$ Dehn filling. 

\begin{itemize}
\item If $\widehat{L}(s_{i}) \geq 14.90\sqrt{n}$ for each $i$, $1 \leq i \leq n$, then the $n$ core geodesics of the filling tori lie in the set of preserved (complex) geodesic lengths.   
\item If $\widehat{L}(s_{i}) \geq 20.76\sqrt{n}$ for each $i$, $1 \leq i \leq n$, then $M$ and $M^{\mu}$ have at least the same $n$ initial values of their respective (complex) length spectra. 
\end{itemize}
\end{cor}

\textbf{Remark:} Corollary \ref{cor:syspreserved2}, and so, the second part of Corollary \ref{cor:syspreservednl}, both require geodesics of length less than $0.015$ in order to get a lower bound on how much of the initial length spectrum is preserved under mutation. The work of Meyerhoff \cite{Me} shows that the Margulis constant for the thick-thin decomposition in dimension $3$ is at least $0.104$. Thus, the geodesics corresponding to the initial length spectrum guaranteed to be preserved under mutation are all contained in the thin parts of these manifolds. Specifically, they must all be cores of solid tori and possibly multiples of these cores. In general, many more geodesics are preserved under mutation. Lemma \ref{lemma:Fgroups} implies that every element of the non-elementary groups $\Gamma_{a}$ and $\Gamma_{b}$, or $\Gamma_{N}$ in the non-separating case, maintains its complex length under mutation. This includes any geodesics that can be homotoped disjoint from the mutation surface. 

%%%%%%%%%%%%%%%%%%%%%%%%%%%%%%%%%%%%%%%%%%%%%%%%%%%

\section{Hyperbolic Pretzel Knots: $\left\{ K_{2n+1} \right\} _{n=2}^{\infty}$}
\label{sec:RT_and_PK}

Here, we construct a specific class of pretzel knots, $\left\{ K_{2n+1} \right\} _{n=2}^{\infty}$. We will be able to show that for each $n \geq 2$, $K_{2n+1}$ generates a large number of mutant pretzel knots whose complements all have the same volume and initial length spectrum.  This section describes pretzel links, their classification, and the basic properties of $\left\{ K_{2n+1} \right\} _{n=2}^{\infty}$.

%%%%%%%%%%%%%%%%%%%%%%%%%%%%%%%%%%%%%%%%%%%%%%%%%%

\subsection{Pretzel Links}
\label{subsec:PL}

We shall describe vertical tangles and see how they can be used to construct pretzel links. Afterwards, we will give a simple classification of pretzel links.

\begin{defn}[Pretzel link]
The \emph{vertical tangles}, denoted by $\frac{1}{n}$, are made of $n$ vertical half-twists, $n \in \mathbb{Z}$, as depicted in figure \ref{verticaltangles}.
A \emph{pretzel link}, denoted $K\left( \frac{1}{q_{1}}, \frac{1}{q_{2}},\ldots, \frac{1}{q_{n}} \right)$, is defined to be the link constructed by connecting $n$ vertical tangles in a cyclic fashion, reading clockwise, with the $i^{th}$-tangle associated with the fraction $\frac{1}{q_{i}}$.  
\end{defn}

\begin{figure}[ht]
\includegraphics[scale=0.65]{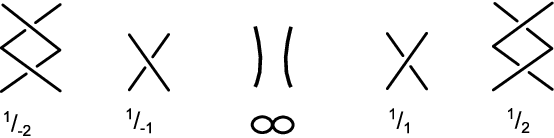}
\caption{Some of the vertical tangles with their associated fractions.}
\label{verticaltangles}
\end{figure}

$K$ in figure \ref{augmentedpretzel} is the pretzel link $K = K( \frac{1}{4}, \frac{1}{7}, \frac{1}{9})$.  Note that, each vertical tangle corresponds with a \textit{twist region} for a knot diagram of a pretzel link. Twist regions are defined at the beginning of Section \ref{subsubsec:AL_and_DF}.

Now, we state the classification of pretzel links, which is a special case of the classification of Montesinos links.  The classification of Montesinos links was originally proved by Bonahon in 1979 \cite{Bo}, and another proof was given by Boileau and Siebenmann in 1980 \cite{BS}.  A proof similar to the one done by Boileau and Siebenmann can be found in \cite[Theorem $12.29$]{BZ}. Here, we state the theorem solely in terms of pretzel links.

\begin{thm} \cite{Bo}
\label{thm:Bo}
The pretzel links $K\left( \frac{1}{q_{1}}, \frac{1}{q_{2}},\ldots, \frac{1}{q_{n}} \right)$ with $n \geq 3$ and $\sum_{j=1}^{n}\frac{1}{q_{j}} \leq n-2$, are classified by the ordered set of fractions $\left(\frac{1}{q_{1}},  \ldots, \frac{1}{q_{n}} \right)$ up to the action of the dihedral group generated by cyclic permutations and reversal of order.
\end{thm}

%%%%%%%%%%%%%%%%%%%%%%%%%%%%%%%%%%%%%%%%%%%%%%%%%%

\subsection{Our Construction}
\label{subsec:pretzelknots}

Consider the pretzel link $K_{2n+1} = K \left( \frac{1}{q_{1}}, \frac{1}{q_{2}}, \ldots, \frac{1}{q_{2n+1}} \right)$, where each $q_{i} > 6$, $q_{1}$ is even, each $q_{i}$ is odd for $i >1$, and $q_{i} \neq q_{j}$ for $i \neq j$. We will always work with the diagram of $K_{2n+1}$ that is depicted below in figure \ref{figure6}.  Each $R_{i}$ in this diagram of $K_{2n+1}$ represents a twist region in which the vertical tangle $\frac{1}{q_{i}}$ takes place. For $n \geq 2$, $K_{2n+1}$ has the properties listed below; details can be found in \cite{Mi}. Though our construction here is slightly different, it still retains all the same key properties listed below.
\begin{enumerate}
\item Each $K_{2n+1}$ is a hyperbolic knot (link with a single component). 
\item This diagram of $K_{2n+1}$ is alternating.
\item This diagram of $K_{2n+1}$ is prime and twist-reduced (definitions can be found in \cite{FP}).
\item Two such pretzel knots are distinct (as knots) if and only if their corresponding complements are non-isometric. This follows from the Gordon--Luecke Theorem \cite{GL} and Mostow--Prasad rigidity. 
\end{enumerate}

\begin{figure}[ht]
\includegraphics[scale=0.55]{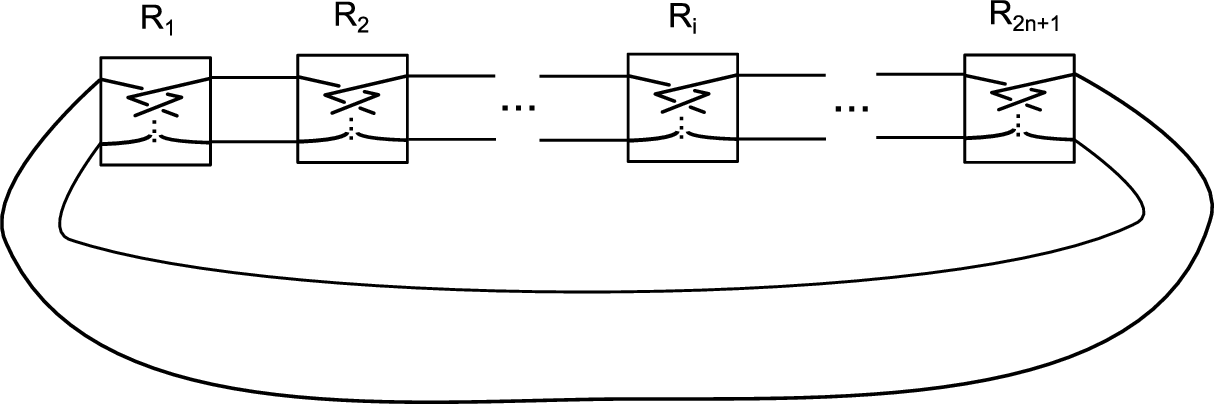}
\caption{The pretzel knot $K_{2n+1}$.  Each twist region $R_{i}$ contains a vertical tangle with $q_{i}$ positive crossings.}
\label{figure6}
\end{figure}

%%%%%%%%%%%%%%%%%%%%%%%%%%%%%%%%%%%%%%%%%

\subsection{Mutations of $K_{2n+1}$ that preserve volume}
\label{sec:Mutations}
In this subsection, we will see how mutations can be useful for preserving the volume of a large class of hyperbolic $3$-manifolds $\left\{M_{2n+1}^{\sigma}\right\}$, with $M_{2n+1}^{\sigma}= \mathbb{S}^{3} \setminus K_{2n+1}^{\sigma}$. $K^{\sigma}_{2n+1}$ is one of our hyperbolic preztel knots constructed in Section \ref{subsec:pretzelknots}, and the upper index $\sigma$ signifies a combination of mutations along Conway spheres, which we will now describe.

Given a $K_{2n+1}$, consider the set $\left\{ (S_{a}, \sigma_{a}) \right\}_{a=1}^{2n}$ where $S_{a}$ is a Conway sphere that encloses only $R_{a}$ and $R_{a+1}$ on one side, and $\sigma_{a}$ is the mutation along $S_{a}$ which rotates about the $y$-axis.  On one of our pretzel knots, such a mutation $\sigma_{a}$ interchanges the vertical tangles $R_{a}$ and $R_{a+1}$, as depicted in figure \ref{SAmutated}. In terms of our pretzel knot vector, such a mutation just switches $\frac{1}{q_{a}}$ and $\frac{1}{q_{a+1}}$.

\begin{figure}[ht]
\includegraphics[scale=0.50]{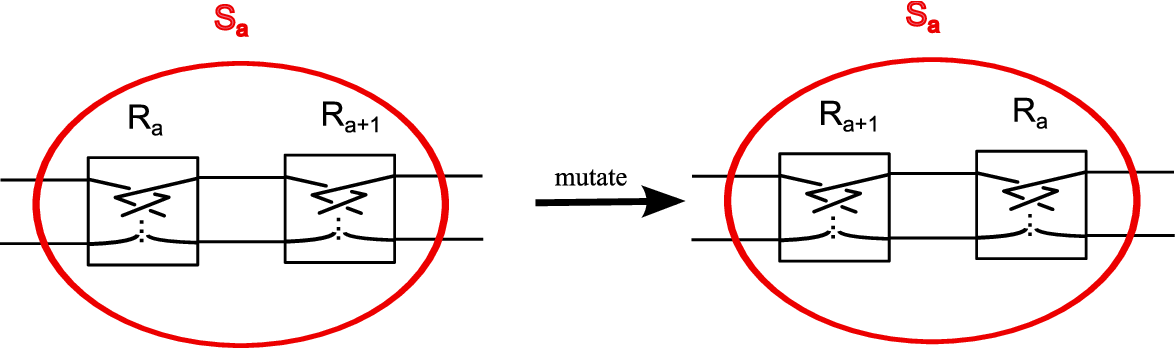}
\caption{Mutation along the Conway sphere $S_{a}$}
\label{SAmutated}
\end{figure}

In \cite{Mi}, we used the following theorem proved by Ruberman to construct many hyperbolic knot complements with the same volume.

\begin{thm} \cite[Theorem $1.3$]{Ru}
\label{thm:Ru1}
Let $\mu$ be any mutation of an incompressible and $\partial$-incompressible hyperelliptic surface in a hyperbolic $3$-manifold $M$.  Then $M^{\mu}$ is also hyperbolic, and $vol(M^{\mu}) = vol(M)$.  
\end{thm}

Ruberman's proof of this theorem requires the hyperelliptic surface $S$ to be isotoped into least area form in order to perform a volume-preserving mutation of a hyperbolic $3$-manifold $M$ along $S$.  This fact will be crucial, considering the conditions for Theorem \ref{thm:gammasep}.

By the proof of \cite[Theorem $2$]{Mi}, for a given $M_{2n+1} = \mathbb{S}^{3} \setminus K_{2n+1}$, performing combinations of mutations along the collection $\left\{ (S_{a}, \sigma_{a}) \right\}_{a=1}^{2n}$ produces a large number of non-isometric hyperbolic knot complements with the same volume, and this number grows as n increases.  Specifically, we have:

\begin{thm} \cite{Mi}
\label{thm:volume}
For each $n \in \mathbb{N}$, $n>2$, there exist $\frac{(2n)!}{2}$ distinct hyperbolic pretzel knots, $\left\{K_{2n+1}^{\sigma}\right\}$, obtained from each other via mutations along the Conway spheres $\left\{ (S_{a}, \sigma_{a}) \right\}$. Furthermore, for each such $n$, 
\begin{itemize}
\item their knot complements have the same volumes, and 
\item $\left(\frac{2n-1}{2}\right)v_{\mathrm{oct}} \leq vol(M^{\sigma}_{2n+1})  \leq \left(4n+2\right)v_{\mathrm{oct}}$, where $v_{\mathrm{oct}} \left(\approx 3.6638\right)$ is the volume of a regular ideal octahedron.
\end{itemize}

\end{thm}

%%%%%%%%%%%%%%%%%%%%%%%%%%%%%%%%%%%%%%%%%%%%%%%%%%

\section{The Geometry of Untwisted Augmented Links}
\label{sec:GEOofUAL}

The goal of this section is to better understand the geometry and topology of our pretzel knots by realizing them as Dehn fillings of untwisted augmented links. Recall that $K_{2n+1} = K \left( \frac{1}{q_{1}}, \frac{1}{q_{2}}, \ldots, \frac{1}{q_{2n+1}} \right)$ with $q_{1}$ even, while the rest are odd and distinct.  We can realize each $K_{2n+1}$ as a Dehn surgery along specific components of a hyperbolic link $L_{2n+1}$. We want to find a lower bound on the normalized length of the Dehn filling slopes along each of these components in order to apply Corollary \ref{cor:syspreservednl}. We also can understand the cusp shape of $K_{2n+1}$ by first studying the cusp shape of this knot as a component of $L_{2n+1}$.  This will be used to determine that these knots are pairwise incommensurable in Section \ref{sec:commensurablity}.  The following analysis will help us determine the properties we are interested in.

%%%%%%%%%%%%%%%%%%%%%%%%%%%%%%%%%%%%%%%%%%%%%%%%%

\subsection{Augmented Links}
\label{subsubsec:AL_and_DF}
First, we will go over some basic properties of knots. We usually visualize a knot by its \emph{diagram}.  A diagram of a knot can be viewed as a $4$-valent planar graph $G$, with over-under crossing information at each vertex.  Here, we will need to consider the number of \textit{twist regions} in a given diagram.  A twist region of a knot diagram is a maximal string of bigons arranged from end to end.  A single crossing adjacent to no bigons is also a twist region. We also care about the amount of twisting done in each twist region.  We describe the amount of twisting in terms of half twists and full twists. A \textit{half twist} of a twist region of a diagram consists of a single crossing of two strands.  A \textit{full twist} consists of two half twists.  Now, we can define augmented links, which were introduced by Adams \cite{Ad} and have been studied extensively by Futer and Purcell in \cite{FP} and Purcell in \cite{Pu3}, \cite{Pu2}. For an introduction to augmented links, we suggest first reading \cite{Pu4}.

\begin{figure}[h]
\includegraphics[scale=0.60]{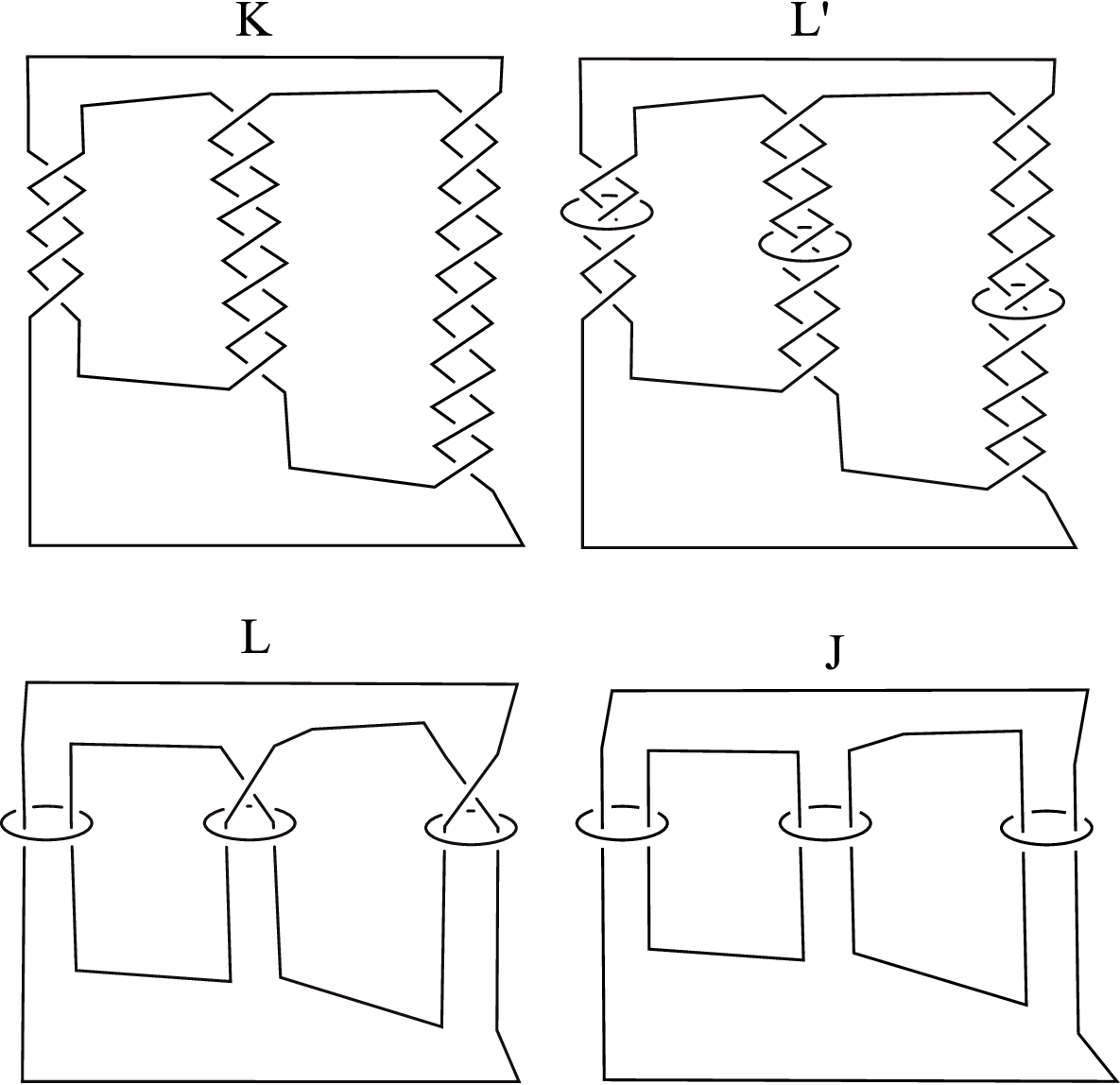}
\caption{Diagrams of a knot $K$ with three twist regions, the augmented link $L'$ with three crossing circles, the untwisted augmented link $L$, and the flat augmented link $J$.}
\label{augmentedpretzel}
\end{figure}

\begin{defn}[Augmented Links]
\label{def:AL}
Given a diagram of a knot or link $K$, insert a simple closed curve encircling each twist region. This gives a diagram for a new link $L'$, which is the \textit{augmented link} obtained from $K$. Obtain a new link $L$ by removing all full twists from each twist region in the diagram of $L'$.  We shall refer to the link $L$ as the \textit{untwisted augmented link}.  Each twist region now has either no crossings or a single crossing. If we remove all of the remaining single crossings from the twist regions, then we form the \textit{flat augmented link}, $J$. 
\end{defn}

The top two diagrams in figure \ref{augmentedpretzel} show a link $K$ with three twist regions and then the corresponding augmented link $L'$. The bottom two diagrams of figure \ref{augmentedpretzel} show the corresponding untwisted augmented link $L$ and flat augmented link $J$. The simple closed curves inserted to augment $K$ are called \textit{crossing circles}. The untwisted augmented link $L$ has a diagram consisting of crossing circle components bounding components from the link.  Near each crossing circle, the link component is embedded in the projection plane if the corresponding twist region contained only full twists.  Otherwise, there is a single half twist. $L$ is made up of two types of components: the crossing circles and the other components coming from the original link $K$.  We shall refer to these other components as the \textit{knot components} of $L$.  When $K$ is a knot, there is a single knot component in $L$, which will be the case for our work.

The $3$-manifolds $S^{3} \setminus L$ and $S^{3} \setminus L'$ actually are homeomorphic.  Performing $t_{i}$ full twists along the punctured disk bound by a crossing circle and then regluing this disk gives a homeomorphism between link exteriors.  Thus, if either $S^{3} \setminus L$ or $S^{3} \setminus L'$ is hyperbolic, then Mostow--Prasad rigidity implies that the two manifolds are isometric.

Next, we shall examine the polyhedral decompositions of certain untwisted augmented links.  We will do this by first examining such structures on the corresponding flat augmented links, which are almost the same, but easier to initially analyze.

%%%%%%%%%%%%%%%%%%%%%%%%%%%%%%%%%%%%%%%%%%%%%%%%%%

\subsection{Ideal Polyhedral Decompositions of Untwisted Augmented Links}
\label{subsec:PDofUAL}
   
The polyhedral decompositions of untwisted augmented link complements have been thoroughly described in \cite{FP}.  This polyhedral decomposition was first described by Agol and Thurston in the appendix of \cite{La}, and many of its essential properties are highlighted in the following theorem.

\begin{thm}
\label{thm:polyprops}
Let $L$ be the untwisted augmented link corresponding to a link $K$.  Assume the given diagram of $K$ is prime, twist-reduced, and $K$ has at least two twist regions. Then $\mathbb{S}^{3} \setminus L$ has the following properties:

\begin{enumerate}
\item $\mathbb{S}^{3} \setminus L$ has a complete hyperbolic structure.
\item This hyperbolic $3$-manifold decomposes into two identical ideal, totally geodesic polyhedra, $I$ and $I'$, all of whose dihedral angles are $\frac{\pi}{2}$.
\item The faces of $I$ and $I'$ can be checkerboard colored, shaded and white.
\item Shaded faces come in pairs on each polyhedron, and they are constructed by peeling apart half of a single $2$-punctured disc bounded by a crossing circle. 
\item White faces come from portions of the projection plane bounded by knot strands. 
\end{enumerate}
\end{thm}

Here, we will briefly describe this decomposition and the resulting circle packings, with emphasis on our untwisted augmented link complements, $N_{2n+1} = \mathbb{S}^{3} \setminus L_{2n+1}$.  We direct the reader to \cite[Sections 6 and 7]{Pu2} for more details on cusp shape analysis of untwisted augmented link complements.

\begin{figure}[ht]
\includegraphics[scale=0.50]{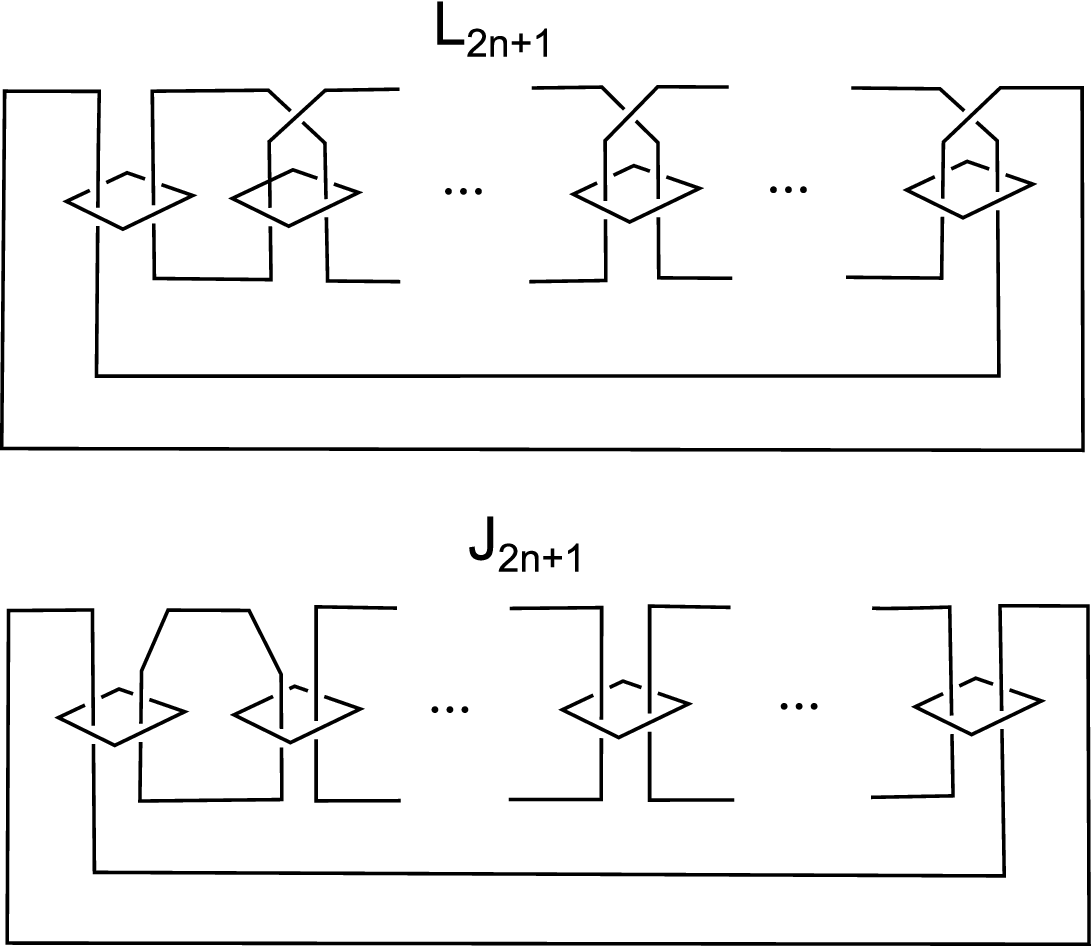}
\caption{The untwisted augmented link $L_{2n+1}$ and the flat augmented link $J_{2n+1}$.}
\label{flatanduntwisted}
\end{figure}

First, consider $\mathbb{S}^{3} \setminus J_{2n+1}$, where $J_{2n+1}$ is the flat augmented link, whose diagram is shown in figure \ref{flatanduntwisted}.  In the diagram of $J_{2n+1}$, the knot strands all lie on the projection plane. To subdivide $\mathbb{S}^{3} \setminus J_{2n+1}$ into polyhedra, first slice it along the projection plane, cutting $\mathbb{S}^{3}$ into two identical $3$-balls.  These identical polyhedra are transformed into ideal polyhedra by collapsing strands of $J_{2n+1}$ to ideal vertices.  These ideal polyhedra have two types of faces: shaded faces and white faces, described in the above theorem.

To go from an ideal polyhedral decomposition of $\mathbb{S}^{3} \setminus J_{2n+1}$ to one for $\mathbb{S}^{3} \setminus L_{2n+1}$, we just have to introduce a half-twist into our gluing at each shaded face where a crossing circle bounds a single twist. Depicted in figure \ref{polydecomp} below is an ideal polyhedral decomposition of the flat augmented link, $J_{2n+1}$.

\begin{figure}[h]
\includegraphics[scale=0.60]{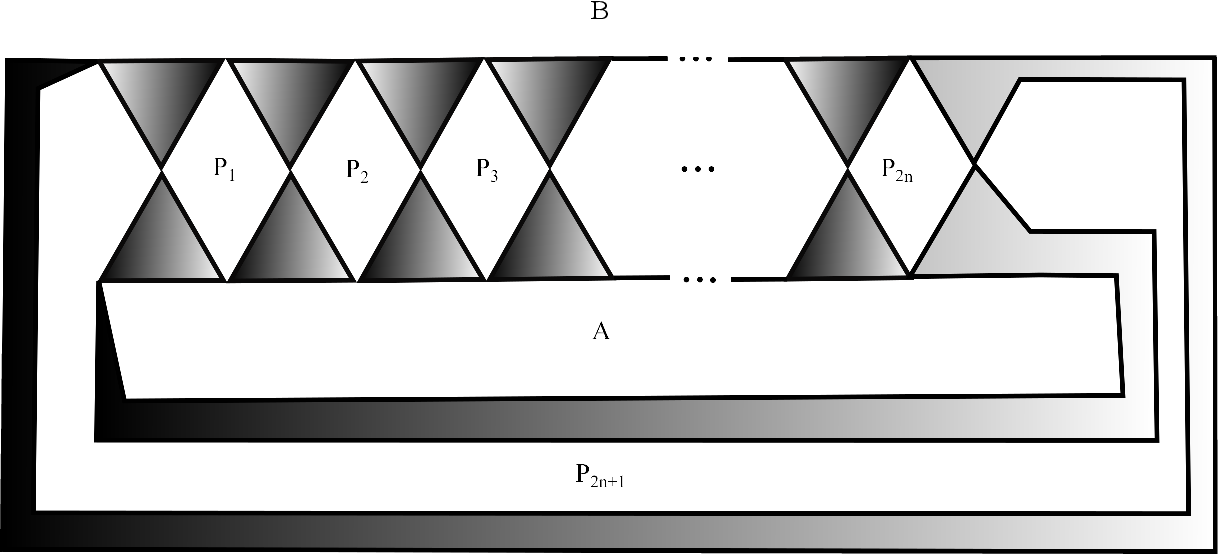}
\caption{The polyhedral decomposition of $J_{2n+1}$.}
\label{polydecomp}
\end{figure}

In \cite[Section 6]{Pu2}, Purcell describes a circle packing associated to the white faces of the polyhedra (and a dual circle packing associated to the shaded faces). Figure \ref{pretzelcuspshape} depicts the circle packing coming from the white faces of the polyhedral decomposition of $J_{2n+1}$.

\begin{figure}[h]
\includegraphics[scale=0.60]{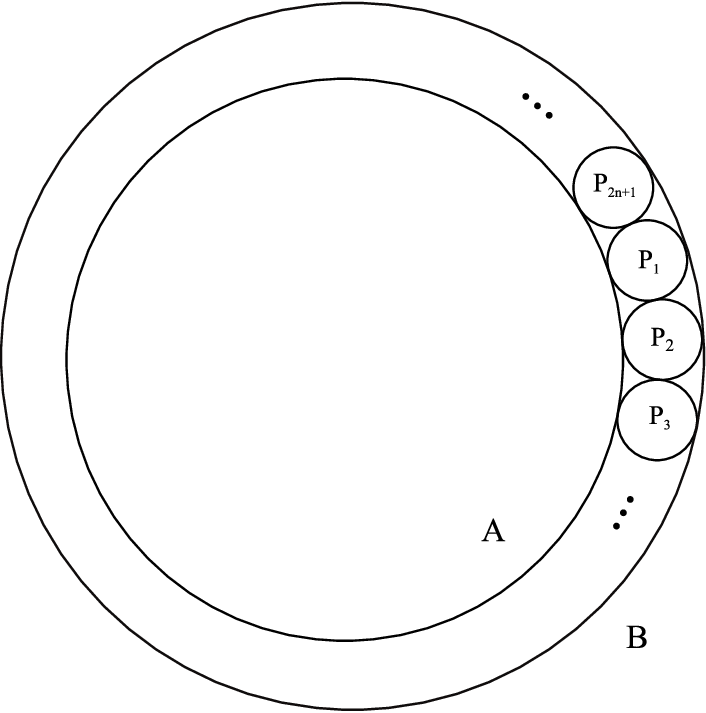}
\caption{The resulting circle packing for $J_{2n+1}$}
\label{pretzelcuspshape}
\end{figure}

The decomposition of $\mathbb{S}^{3} \setminus L_{2n+1}$ is determined by this circle packing. First, slice off half-spaces bounded by geodesic hemispheres in $\mathbb{H}^{3}$ corresponding to each circle in the circle packing.  These give the geodesic white faces of the polyhedron. The shaded faces are obtained by slicing off hemispheres in $\mathbb{H}^{3}$ corresponding to each circle of the dual circle packing. Finally, we just need to make sure we glue up most of the shaded faces with a half-twist. Only the two shaded faces corresponding to the first twist region are glued up without a half-twist.

A careful analysis of this polyhedral decomposition also leads to a canonical method for cusp expansion. Given any $\mathbb{S}^{3} \setminus L_{2n+1}$, we have $2n+2$ cusps corresponding to the $2n+2$ components of the link $L_{2n+1}$ (the $2n+1$ crossing circles and the single knot component). We initially start with disjoint horoball neigborhoods of these cusps and then follow the expansion instructions described in \cite[Section $3$]{FP}: given any ordering of the cusps, expand them one at a time until a horoball neighborhood $C$ meets another horoball neighborhood or $C$ meets the \textit{midpoint} of some edge of our polyhedral decomposition. See \cite[Definition $3.6$]{FP} for the definition of midpoint in this context. This choice of cusp expansion results in the following theorem, which is now stated in terms of our untwisted augmented link complements.

\begin{thm}\cite{FP}
\label{thm:cuspexpansion}
Given any $\mathbb{S}^{3} \setminus L_{2n+1}$, expand the cusps as described above. This results in a unique horoball packing where each boundary horosphere of a horoball neighborhood meets the midpoint of every edge asymptotic to its ideal point. 
\end{thm}

The fact that this cusp expansion is unique will be essential for analyzing cusp neighborhoods and horoball packings in Section \ref{subsec:NLonC} and in Proposition \ref{prop:no_rigid}.

%%%%%%%%%%%%%%%%%%%%%%%%%%%%%%%%%%%%%%%%%%%%%%%%

\subsection{Normalized Lengths on Cusps}
\label{subsec:NLonC}

For this section, we will specialize our analysis to just our pretzel knot complements $M_{2n+1} = \mathbb{S}^{3} \setminus K_{2n+1}$ which result from Dehn filling the $2n+1$ crossing circles, $\left\{ C_{i} \right\}_{i=1}^{2n+1}$, of $N_{2n+1} = \mathbb{S}^{3} \setminus L_{2n+1}$. Recall that $K_{2n+1}$ has $2n+1$ twist regions with $q_{i}$ crossings in the $i^{th}$ twist region, and in $L_{2n+1}$, exactly $2n$ of these crossing circle enclose a single crossing since $2n$ of our $q_{i}$ are odd. To apply Corollary \ref{cor:syspreservednl}, we will need to examine normalized lengths of particular slopes on the cusps in $N_{2n+1}$ corresponding to crossing circles. In \cite[Proposition 6.5]{Pu2}, Purcell gives the general case for providing bounds on the normalized lengths $\widehat{L}(s_{i})$ of Dehn filling crossing circles of an untwisted augmented link. In the general case, re-inserting $q_{i}$ crossings gives $\widehat{L}(s_{i}) \geq \sqrt{q_{i}}$. By restricting to untwisted augmented links corresponding to hyperbolic pretzel knots, we are able to provide a substantial improvement on this bound, highlighted in the proposition given below.

\begin{prop}
\label{prop:NL}
On the cusps of $N_{2n+1}$ corresponding to crossing circles, we have the following normalized lengths:
Let $s_{i}$ be the slope such that Dehn filling $N_{2n+1}$ along $s_{i}$ re-inserts the $q_{i}-1$ or $q_{i}$ crossings at that twist region.  Then $\widehat{L}\left(s_{i}\right) \geq   \sqrt{\frac{(2n-1)(1+q_{i}^{2})}{4n}}$. In particular, if $n \geq 2$, we have that $\widehat{L}(s_{i}) \geq \sqrt{ \frac{3(1+q_{i}^{2})}{8}} $.
\end{prop} 

\begin{proof}
Pictured in figure \ref{pretzelcuspshape} is a circle packing for $J_{2n+1}$ coming from the white faces. There also exists a circle packing for the shaded faces, which is dual to the circle packing coming from the white faces. These two circle packings also determine the same circle packings for $L_{2n+1}$ since the only difference between $L_{2n+1}$ and $J_{2n+1}$ is how the two ideal polyhedra are glued together. Much of what follows in the next two paragraphs is done in \cite[Sections $2$ and $3$]{FP}. In their work, the cusp shapes are analyzed with respect to any augmented link, while we will specialize to our $L_{2n+1}$.  

First, let us recall our polyhedra obtained in Section \ref{subsec:PDofUAL}.  Each cusp will be tiled by rectangles given by the intersection of the cusp with the totally geodesic white and shaded faces of the polyhedra.  Two opposite sides of each of these rectangles come from the intersection of the cusp with shaded faces of the polyhedra (corresponding with the $2$-punctured disc in the diagram of $L_{2n+1}$), and the other two sides from white faces.  Call these sides shaded sides and white sides, respectively. We can make an appropriate choice of cusp neighborhoods as in Theorem \ref{thm:cuspexpansion}. This allows us to consider the geometry of our rectangles tiling a cusp.  

Our crossing circle cusp is tiled by two rectangles, each rectangle corresponding with a vertex in one of the polyhedra.  In terms of our circle packing of $\mathbb{S}^{2}$, this vertex corresponds with a point of tangency of two circles.  Consider the point of tangency given by $P_{i} \cap P_{i+1}$, which corresponds to one of the two identical rectangles making up the crossing circle cusp $C_{i+1}$. By the rotational symmetry of the circle packing in figure \ref{pretzelcuspshape}, all of these rectangles (along with their circle packings) are in fact isometric.  Thus, taking a step along a shaded side will be the same for any such rectangle, and similarly for stepping along a white side. Let $s$ represent taking one step along a shaded face and $w$ represent taking one step along a white face.  Each torus cusp, $T$, has universal cover $\tilde{T} = \mathbb{R}^{2}$.  $\tilde{T}$ contains a rectangular lattice coming from the white and shaded faces of our polyhedron. We let $(s,w)$ be our choice of basis for this $\mathbb{Z}^{2}$ lattice.   

Now, we shall examine the normalized length in terms of our longitudes and meridians of the cusps corresponding to crossing circles.  Lemma $2.6$ from \cite{FP} tells us that the meridian is given by $w \pm s$ when there is a half-twist, and the meridian is $w$ without the half-twist. In either case, the longitude is given by $2s$. When $q_{i}$ is odd, $\frac{q_{i} - 1}{2}$ full twists were removed in constructing $L_{2n+1}$, so the surgery slope for the $i^{th}$ crossing circle will be $(1, \frac{q_{i} - 1}{2})$.  Thus, the slope $s_{i}$ is given by $(w \pm s) \pm \frac{q_{i}-1}{2} (2s) = w \pm q_{i}s$, when $q_{i}$ is odd.  For the single even $q_{i}$, the surgery slope is $(1, \frac{q_{i}}{2})$ and the slope is given by $w \pm \frac{q_{i}}{2} (2s) = w \pm q_{i}s$; see \cite[Theorem 2.7]{FP}. In either case, the normalized length of $s_{i}$ is:
\begin{center}
$\widehat{L}(s_{i}) = \frac{ \sqrt{ \ell(w)^{2} + q_{i}^{2}\ell(s)^{2} } } {\sqrt{ 2\ell(w)\ell(s)}}$.
\end{center}

Here, $\ell(w)$ and $\ell(s)$ denote the lengths of $w$ and $s$ respectively, on our choice of cusp neighborhoods. To bound the normalized length, we need to bound $\ell(w)$ and $\ell(s)$.  We shall use our circle packing to obtain such bounds. Consider the tangency given by $P_{i} \cap P_{i+1}$, which corresponds to one of the two rectangles making up our cusp. Note that, $P_{i}$ is also tangent to circles $P_{i-1}$, $A$, and $B$, while $P_{i+1}$ is also tangent to $P_{i+2}$, $A$, and $B$ ($i$ values taken mod $2n+1$).  Apply a M\"{o}bius transformation taking $P_{i} \cap P_{i+1}$ to infinity. This takes the two tangent circles $P_{i}$ and $P_{i+1}$ to parallel lines, as in figure \ref{pretzelcuspshapes2}. This also gives the similarity structure of the rectangle under consideration.  Our choice of cusp neighborhoods results in $\ell(s) =1$.  This makes the circles $A$ and $B$ lying under the dashed lines in figure \ref{pretzelcuspshapes2} have diameter $1$.  Since circles in our circle packing can not overlap, this forces $\ell(w) \geq 1$.  Note that, the dashed lines come from our dual circle packing corresponding to shaded faces.

\begin{figure}[h]
\includegraphics[scale=0.75]{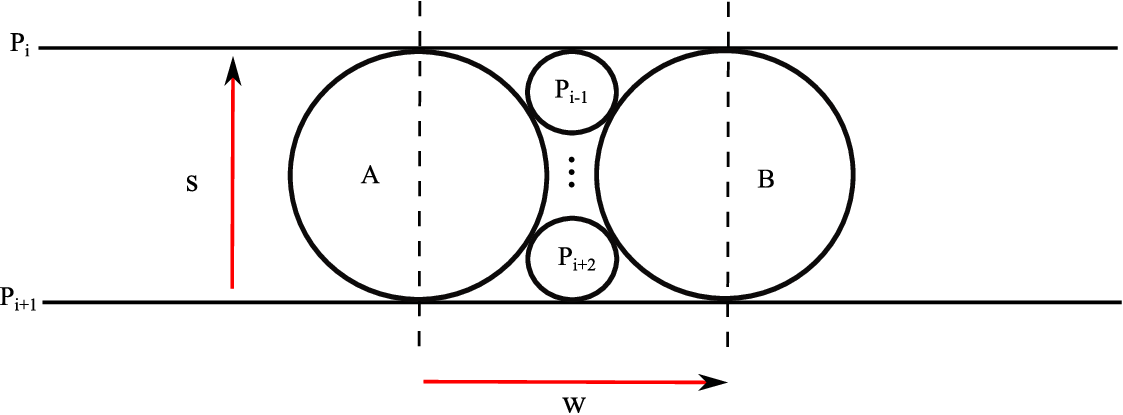}
\caption{The cusp shape of one of the rectangles tiling our crossing circle cusp.  This rectangle is determined by sending the tangency point $P_{i} \cap P_{i+1}$ to $\infty$}.
\label{pretzelcuspshapes2}
\end{figure}

Now, we just need to find an upper bound for $\ell(w)$. Again, consider figure \ref{pretzelcuspshapes2}.  Since $P_{j}$ is tangent to $A$, $B$, $P_{j-1}$, and $P_{j+1}$ for $1 \leq j \leq 2n+1$, all the circles $P_{1}, \dots, P_{i-1}, P_{i+2}, \dots, P_{2n+1}$ lie in between our parallel lines and in between $A$ and $B$, stacked together as depicted in figure \ref{pretzelcuspshapes2} to meet our tangency conditions. Notice, that this circle packing of one of these rectangles has two lines of symmetry: the line $l_{w}$ going through the two $w$ sides in their respective midpoints, and the line $l_{s}$ going through the two $s$ sides in their respective midpoints. $l_{w}$ is a translate of $s$ and $l_{s}$ is a translate of $w$. Reflecting across either of these lines preserves our circle packing. In particular, $l_{w}$ must intersect each $P_{j}$, $j \neq i, i+1$ in a diameter. Let $D(P_{j})$ denote the diameter of circle $P_{j}$.  Then $\displaystyle\sum\limits_{j \neq i, i+1} D(P_{j}) = l(s) = 1$.

Next, the fact that we have symmetries about both $l_{s}$ and $l_{w}$ and an odd number of $P_{j}$ packed in between $A$ and $B$ implies that one of our $P_{j}$'s is centered at $l_{s} \cap l_{w}$. Call this circle $P_{j}^{\ast}$. Note that, $l_{s}$ intersects $A$, $B$, and $P_{j}^{\ast}$ in their respective centers. Thus, $\ell(w) = \ell(l_{s}) = \frac{D(A)}{2} + \frac{D(B)}{2} + D(P_{J}^{\ast}) = 1 + D(P_{J}^{\ast})$.

Now, we claim that $P_{j}^{\ast}$ has the minimal diameter amongst $P_{j}$, $j \neq i, i+1$. This follows from our tangency conditions: each such $P_{j}$ must be tangent to both $A$ and $B$.  The diameter of $P_{j}^{\ast}$ obviously minimizes the distance between $A$ and $B$. For any other $P_{j}$, consider the line $l_{j}$ in $P_{j}$ from $P_{j} \cap A$ to $P_{j} \cap B$.  Then we have that $D(P_{j}^{\ast}) \leq \ell(l_{j}) \leq D(P_{j})$. The first inequality holds because $D(P_{j}^{\ast})$  minimizes distance from $A$ to $B$, while the second inequality is obviously true for any circle. So, $D(P_{j}^{\ast})$ must be the smallest such diameter.

Finally, we have $1 = \ell(s) = \displaystyle\sum\limits_{j \neq i, i+1} D(P_{j})  \geq \displaystyle\sum\limits_{j \neq i, i+1} D(P_{j}^{\ast}) = (2n-1)D(P_{j}^{\ast})$, which implies that $D(P_{j}^{\ast}) < \frac{1}{2n-1}$. This helps give us the desired upper bound on $\ell(w)$:
\begin{center}
$\ell(w) = \ell(l_{s}) = 1 + D(P_{J}^{\ast}) \leq 1 + \frac{1}{2n-1} = \frac{2n}{2n-1}$.
\end{center}

With these bounds, we have that 
\begin{center}
$\widehat{L}(s_{i}) = \frac{ \sqrt{ \ell(w)^{2} + q_{i}^{2}\ell(s)^{2} } } {\sqrt{ 2\ell(w)\ell(s)}} = \frac{ \sqrt{ \ell(w)^{2} + q_{i}^{2} } }{ \sqrt{2\ell(w)}} \geq \frac { \sqrt {1+q_{i}^{2} } } {\sqrt{2\ell(w)}} \geq \frac { \sqrt {1+q_{i}^{2} } } {\sqrt{\frac{4n}{2n-1}}} = \sqrt{\frac{(2n-1)(1+q_{i}^{2})}{4n}}$. 
\end{center} 

In particular, if $n \geq 2$, we have that $\widehat{L}(s_{i}) \geq \sqrt{ \frac{3(1+q_{i}^{2})}{8}} $. 
\end{proof}

We will also need to analyze the cusp shape of the one cusp $C$ corresponding to the knot component of $L_{2n+1}$.  Such an analysis will play an important role in determining that our knot complements are not commensurable with one another.  We will see that the tiling of the cusp $C$ by rectangles which come from truncating certain vertices of our ideal polyhedral decomposition has a number of nice properties, highlighted in the following proposition.

\begin{prop}
\label{prop:Knotcusp}
Let $C$ be the cusp corresponding to the knot component of $L_{2n+1}$. This cusp has the following properties:
\begin{enumerate}
\item There are $4(2n+1)$ rectangles tiling this cusp, half of which come from each ideal polyhedron.
\item This cusp shape is rectangular (and not a square).
\item All of these rectangles, along with their circle packings, are isometric to one another.
\end{enumerate}
\end{prop}

\begin{proof}
Theorem \ref{thm:cuspexpansion} gives us an appropriate choice of cusp neighborhoods, which allows us to fix the geometry of our cusp $C$.

$\textbf{(1):}$ Consider the ideal polyhedral decomposition in figure \ref{polydecomp} for $J_{2n+1}$. There are $2n+1$ disks corresponding to crossing circles, and we peel each of these disks apart to obtain $2(2n+1)$ shaded faces on each polyhedron. For each shaded face, there are two vertices corresponding to rectangles that tile the knot component cusp $C$; specifically, the two vertices meeting $A$ or the two vertices meeting $B$, depending on the face. Since each of these vertices is shared by exactly two shaded faces, we obtain $2(2n+1)$ rectangles from each polyhedron, or $4(2n+1)$ total such rectangles. $L_{2n+1}$ admits the same polyhedral decomposition as $J_{2n+1}$; the only difference is that the gluing along shaded faces might change the gluing of the polyhedron.

$\textbf{(2):}$ This holds if there are no half-twists under any of the crossing circles, as in $J_{2n+1}$; see \cite[Section 2]{FP}.  However, $L_{2n+1}$ has $2n$ half-twists in its diagram. A half-twist shifts the gluing of the rectangles making up the cusp.  Since $K$ is a knot, it must go through each crossing circle twice, and so, it will pass through an even number of half-twists. Thus, from Lemma $2.6$ in \cite{FP}, the fundamental domain for this torus is given by the meridian $2s$ and the longitude $2(2n+1)w + 2ks$, for some integer $k$.  By a change of basis, we can see that this cusp shape is once again rectangular. Note that, this fundamental domain is not a square since $\ell(s) = 1$ and $1 < \ell(w) < 2$. 

$\textbf{(3):}$ Consider the circle packing depicted in figure \ref{pretzelcuspshape}.  The rectangles tiling our cusp $C$ come from mapping $P_{i} \cap A$ to $\infty$ or mapping $P_{i} \cap B$ to $\infty$ for $i = 1, \ldots, 2n+1$. By the rotational symmetry of this circle packing, any $P_{i} \cap A$ and $P_{j} \cap A$ will determine isometric rectangles, and similarly for $P_{i} \cap B$ and $P_{j} \cap B$. In fact,  $P_{i} \cap A$ and $P_{j} \cap B$ will also determine isometric rectangles. The circle packings of these rectangles are exactly the same except the roles of $A$ and $B$ have been switched; see figure \ref{pretzelcuspshapes3}. We can also see that $P_{i} \cap A$ and $P_{j} \cap B$ determine isometric rectangles by considering the reflection through the circle running through the $P_{i} \cap P_{j}$. This reflection gives a symmetry exchanging $A$ and $B$. 
\end{proof}

\begin{figure}[h]
\includegraphics[scale=0.75]{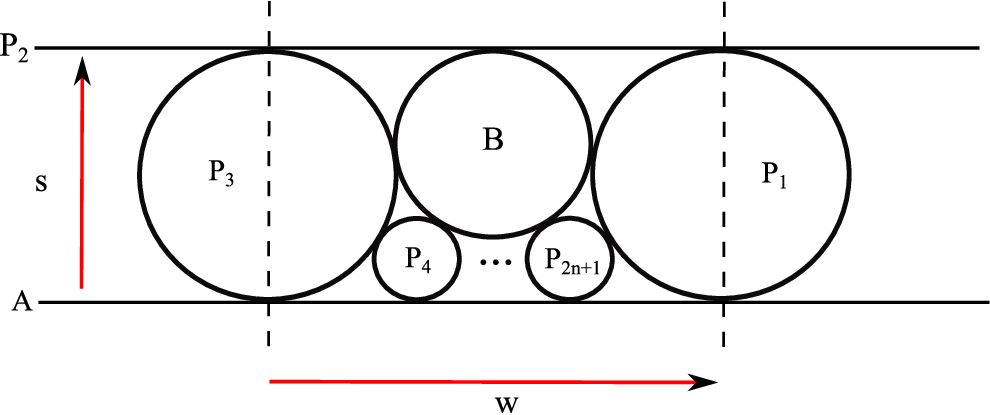}
\caption{The cusp shape of any one of the rectangles tiling our knot cusp $C$.} 
\label{pretzelcuspshapes3}
\end{figure}

Without loss of generality, we will assume any such rectangle coming from the tiling of our knot cusp looks like the one depicted in figure \ref{pretzelcuspshapes3}, i.e., we assume $P_{i} \cap A$ is mapped to $\infty$.

\begin{lemma}
\label{lemma:rectanglesize}
Let $R$ be any rectangle from the tiling of $C$. Let $P_{j}^{\ast}$ be the smallest such $P_{j}$ in the circle packing of this rectangle. Then for all $n \geq 2$, the circle packing of $R$ has the following size bounds:  
\begin{enumerate}
\item $\ell(s) =1$ and $ 1< \ell(w) < 2$, 
\item $\frac{n-2}{n-1} < D(B) < 1$, 
\item $D(B) > \frac{1}{2}$,
\item $D(P_{j}^{\ast}) < \frac{1}{n-1}$.
\end{enumerate}
\end{lemma}

\begin{proof}
As before, our choice of cusp neighborhood results in $\ell(s) = 1$.  Then $D(P_{1}) = D(P_{3}) = 1$. We will assume our rectangle is the one depicted in figure \ref{pretzelcuspshapes3}. By part $3$ of Proposition \ref{prop:Knotcusp}, all such rectangles tiling our cusp, along with their circle packings, are isometric to this one, up to relabelling.

First, we claim that for any $L_{2n+1}$, $1< \ell(w) < 2$. The lower bound follows from the fact that $D(P_{1}) = D(P_{3}) = 1$, and $P_{1}$ and $P_{3}$ can not be tangent to one another.  If $\ell(w) > 2$, then $D(B) > 1$ in order to be tangent to both $P_{1}$ and $P_{3}$.  However, since $\ell(s) =1 $, $B$ would not be tangent to $A$ and $P_{2}$. If $\ell(w) = 2$, then $D(B) =1$ in order to meet its tangency conditions.  Since $\ell(s) =1$, $B$ must separate our rectangle into two parts, one to the right of $B$ and one to the left of $B$.  This violates the tangency conditions of the $P_{j}$, for $j = 4 , \dots, 2n+1$.  So, $1< \ell(w) < 2$ and $D(B) < 1$.

Take the vector $w$ and translate it vertically so it intersects $P_{j}^{\ast}$ in its center.  This line will intersect all the $P_{j}$ in some segment $l(P_{j})$, which must be at least as large as $D(P_{j}^{\ast})$. This can easily be seen by translating $P_{j}^{\ast}$ horizontally along this line so that its point of tangency with $A$ is $P_{j} \cap A$.  Note that, there are exactly $2n-2$ circles $\left\{ P_{j} \right\}_{j=4}^{2n+1}$ packed under $B$. This gives the following inequality:
\begin{center}
$2 > \ell(w) > \displaystyle\sum\limits_{j=4}^{2n+1} l(P_{j}) \geq \displaystyle\sum\limits_{j = 4}^{2n+1} D(P_{j}^{\ast}) = (2n-2)D(P_{j}^{\ast})$.
\end{center}
This gives us that $D(P_{j}^{\ast}) < \frac{2}{2n-2} = \frac{1}{n-1}$.

Now, for any such $j$, $D(B) + D(P_{j}) \geq 1$. Combining with the previous result, we have that
\begin{center}
$D(B) \geq 1 - D(P_{j}^{\ast}) > 1 - \frac{1}{n-1} = \frac{n-2}{n-1}$,
\end{center}
as desired.

Finally, we need to show that $D(B) > \frac{1}{2}$. This is already true if $n>2$ since $\frac{n-2}{n-1} < D(B)$. So, assume $n=2$, which means there are exactly two circles, $P_{4}$ and $P_{5}$, packed under $B$. Suppose $D(B) \leq \frac{1}{2}$. Then $D(P_{4}) > \frac{1}{2}$ since $D(B) + D(P_{4}) > 1$.  Also, $\ell(w) \leq D(B) + \frac{D(P_{1})}{2} +  \frac{D(P_{3})}{2}  = \frac{3}{2}$.  Take the vector $w$ and translate it vertically so that it intersects $P_{4}$ in its center, and take the vector $s$ and translate it horizontally so that it intersects $P_{3}$ in its center. We shall still refer to the translates of these vectors as $w$ and $s$, respectively.  Now consider the right triangle with vertices at the center of $P_{3}$, $w \cap s$, and the left end point of $P_{4} \cap w$. The hypotenuse, $c$, of this triangle has length at least $\frac{1}{2}$ since $\frac{D(P_{3})}{2} = \frac{1}{2}$. The height, $a$, has length less than $\frac{1}{4}$ since  $\frac{D(P_{3})}{2} = \frac{1}{2}$ and $\frac{D(P_{4})}{2} \geq \frac{1}{4}$.  The base, $b$, has length less than $\frac{1}{4}$ since $\frac{\ell(w)}{2} \leq \frac{3}{4}$ and $D(P_{4}) > \frac{1}{2}$. This gives us that $\frac{1}{4} \leq \ell(c)^{2} = \ell(a)^{2} + \ell(b)^{2} \leq \frac{1}{16} + \frac{1}{16} = \frac{1}{8}$, which is a contradiction.  Thus, $D(B) > \frac{1}{2}$.

\end{proof}

%%%%%%%%%%%%%%%%%%%%%%%%%%%%%%%%%%%%%%%%%%%%%

\section{Commensurability classes of hyperbolic pretzel knot complements}
\label{sec:commensurablity}

Recall that two hyperbolic $3$-manifolds $M_{1} = \mathbb{H}^{3} / \Gamma_{1}$ and $M_{2} = \mathbb{H}^{3} / \Gamma_{2}$ are called \textit{commensurable} if they share a common finite-sheeted cover.  In terms of fundamental groups, this definition is equivalent to $\Gamma_{1}$ and a conjugate of $\Gamma_{2}$ in PSL$(2, \mathbb{C})$ sharing some finite index subgroup.  The \textit{commensurability class} of a hyperbolic $3$-manifold $M$ is the set of all $3$-manifolds commensurable with $M$.  

We are interested in the case when $M = \mathbb{S}^{3} \setminus K$, where $K$ is a hyperbolic knot.  It is conjectured in \cite{ReWa} that there are at most three knot complements in the commensurability class of a hyperbolic knot complement.  In particular, Reid and Walsh show that when $K$ is a hyperbolic $2$-bridge knot, then $M$ is the only knot complement in its commensurability class.  Their work provides  criteria for checking whether or not a hyperbolic knot complement is the only knot complement in its commensurability class.  Specifically, we have the following theorem coming from Reid and Walsh's work in \cite[Section 5]{ReWa}; this version of the theorem can be found at the beginning of \cite{MM}.

\begin{thm}
\label{thm:comm_knots}
Let $K$ be a hyperbolic knot in $\mathbb{S}^{3}$.  If $K$ admits no hidden symmetries, has no lens space surgery, and admits either no symmetries or else only a strong inversion and no other symmetries, then $\mathbb{S}^{3} \setminus K$ is the only knot complement in its commensurability class.
\end{thm}

Macasieb and Mattman use this criterion in \cite{MM} to show that for any hyperbolic pretzel knot of the form $K\left( \frac{1}{-2}, \frac{1}{3}, \frac{1}{n} \right)$, $n \in \mathbb{Z} \setminus \left\{7\right\}$, its knot complement $\mathbb{S}^{3} \setminus K\left( \frac{1}{-2}, \frac{1}{3}, \frac{1}{n} \right)$ is the only knot complement in its commensurability class.  The main challenge in their work was showing that these knots admit no \textit{hidden symmetries}. 

\begin{defn}
\label{defn:hiddensym}
Let $\Gamma$ be a finite co-volume Kleinian group. The \textit{normalizer} of $\Gamma$ is 

\begin{center}
$N(\Gamma) = \left\{ g \in \text{PSL}(2, \mathbb{C}) : g\Gamma g^{-1} = \Gamma \right\}$.  

\end{center}
The \textit{commensurator} of $\Gamma$ is 
\begin{center}
$C(\Gamma) = \left\{ g \in \text{PSL}(2, \mathbb{C}) : \left|\Gamma : \Gamma \cap g\Gamma g^{-1} \right| < \infty \: \text{and} \left|g \Gamma g^{-1} : \Gamma \cap g^{-1}\Gamma g \right| < \infty \right\}$.
\end{center}

If $N(\Gamma)$ is strictly smaller than $C(\Gamma)$, then $\Gamma$ and $\mathbb{H}^{3} / \Gamma$ are said to have \textit{hidden symmetries}. If $\mathbb{H}^{3} / \Gamma \cong \mathbb{S}^{3} \setminus K$, then we also say that $K$ admits hidden symmetries. 

\end{defn}

Here, we would also like to apply Reid and Walsh's criterion to show that our hyperbolic pretzel knot complements are the only knot complements in their respective commensurability classes.  The following proposition immediately takes care of symmetries and lens space surgeries.  Given a knot $K \subset \mathbb{S}^{3}$, $K$ admits a \textit{strong inversion} if there exists an involution $t$ of $(\mathbb{S}^{3}, K)$ such that the fixed point set of $t$ intersects the knot in exactly two points.

\begin{prop}
\label{prop:no_surg_or_sym}
Let $M  = \mathbb{S}^{3} \setminus K$, where $K = K\left( \frac{1}{q_{1}}, \frac{1}{q_{2}},\ldots, \frac{1}{q_{n}} \right)$ is a hyperbolic pretzel knot with all $q_{i}$ distinct, exactly one $q_{i}$ even, and $K \neq K\left( \frac{1}{-2}, \frac{1}{3}, \frac{1}{7} \right)$.  Then $M$ admits no lens space surgeries, and a strong inversion is its only symmetry. In particular, any $M_{2n+1}^{\sigma}$ admits no lens space surgeries, and a strong inversions is its only symmetry.  
\end{prop}  

\begin{proof}
All pretzel knots admitting lens space surgeries have been classified by Ichihara and Jong in \cite{IJ}, and this classification is also implied by the work of Lidman and Moore in \cite{LiMo}.  Both works show that the only hyperbolic pretzel knot that admits any lens spaces surgeries is $K\left( \frac{1}{-2}, \frac{1}{3}, \frac{1}{7} \right)$.  

To deal with symmetries, we first note that the work of Boileau and Zimmermann \cite{BoZi} implies that Sym$(\mathbb{S}^{3}, K) = \mathbb{Z}_{2}$. It is easy to see that the one non-trivial symmetry of any $K$ is a strong inversion.  Consider the knot diagram of $K_{2n+1}^{\sigma}$ as shown in figure \ref{figure6}.  Recall that exactly one twist region $R_{i}$ has an even number of crossings. Consider the involution of our knot in $\mathbb{S}^{3}$ whose axis cuts directly through the middle of all of our twist regions.  This involution will intersect $K_{2n+1}^{\sigma}$ in exactly two points, always inside the one twist region with an even number of crossings.  In the other twist regions, this axis will miss the knot, passing in between two strands at a crossing. This process for finding the strong involution generalizes to any pretzel knot $K$ with exactly one $q_{i}$ even.  
\end{proof}

It remains to rule out hidden symmetries.  In \cite{MM}, Macasieb and Mattman do this by arguing that the invariant trace field of any $K\left( \frac{1}{-2}, \frac{1}{3}, \frac{1}{n} \right)$ has neither $\mathbb{Q}(i)$ nor $\mathbb{Q}(\sqrt{-3})$ as a subfield.  This criterion for the existence of hidden symmetries is supplied by Neumann and Reid \cite{NeRe}.  Here, we use a geometric approach to show that our knots do not admit hidden symmetries.  We will also use a criterion for the existence of hidden symmetries provided by Neumann and Reid in \cite{NeRe}, stated below.

\begin{prop} \cite[Proposition 9.1]{NeRe}
\label{lemma:no_hidden_sym}
Let $\mathbb{H}^{3} / \Gamma$ be a hyperbolic knot complement which is not the figure-$8$ knot complement. Then $\mathbb{H}^{3} / \Gamma$ admits hidden symmetries if and only if $\mathbb{H}^{3}/ C(\Gamma)$ has a rigid Euclidean cusp cross-section.
\end{prop}

The orientable rigid Euclidean orbifolds are $\mathbb{S}^{2}(2,4,4)$, $\mathbb{S}^{2}(3,3,3)$, and $\mathbb{S}^{2}(2,3,6)$, and are named so because their moduli spaces are trivial. The following proposition will imply that our hyperbolic pretzel knot complements do not admit hidden symmetries, and so, they are the only knot complements in their respective commensurability classes. In what follows, $\mathbb{H}^{3} = \left\{(x,y,z) | z>0\right\}$.

\begin{prop}
\label{prop:no_rigid}
For all  $n \geq 2$ and $q_{i}$ sufficiently large, the hyperbolic knot complement $M = \mathbb{S}^{3} \setminus K = N_{2n+1}\left( (1, q_{1}) , \dots, (1, q_{2n+1}) \right)$ admits no hidden symmetries.
\end{prop} 

\begin{proof}
We will show that any such hyperbolic knot complement does not cover a $3$-orbifold that admits a rigid cusp $2$-orbifold, and so, by Proposition \ref{lemma:no_hidden_sym}, these knot complements admit no hidden symmetries.  First, we shall analyze the cusp of $N_{2n+1}$ corresponding to the knot component of $L_{2n+1}$, and then expand this analysis to the cusp shape of any such $M$. In particular, we will prove that this cusp of $N_{2n+1}$ does not cover any rigid $2$-orbifold. This is accomplished by showing that the horoball packing corresponding to this cusp does not admit an order three or order four rotational symmetry. Then, by taking sufficiently long Dehn surgeries along all of the crossing circles of $L_{2n+1}$, we can make sure that the cusp of $M$ also does not cover any rigid $2$-orbifold.

Throughout this proof, let $C$ denote the cusp of $N_{2n+1}$ that corresponds to the knot component of $L_{2n+1}$.  Lift to $\mathbb{H}^{3}$ so that one of the lifts of the cusp $C$ is a horoball centered at $\infty$, denoted $H_{\infty}$.  There will be a collection of disjoint horoballs in $\mathbb{H}^{3}$ associated with each cusp in $N_{2n+1}$.  We expand our horoballs according to the procedure given by Theorem \ref{thm:cuspexpansion}.  Specifically, we pick an order for our cusps, and expand the horoball neighborhood of a cusp until it either meets another horoball or meets the \textit{midpoint} of some edge of one of the polyhedra; see \cite[Definition $3.6$]{FP} .  This procedure allows us to expand $H_{\infty}$ to height $z=1$, since any other horoballs will have diameter at most $1$ under these expansion instructions; see \cite[Theorem $3.8$]{FP}.  We shall refer to a horoball of diameter $1$ as a \textit{maximal horoball}.  This procedure from \cite[Theorem $3.8$]{FP} results in maximal horoballs sitting at each vertex of a rectangle tiling our cusp cross-section $C$.

\begin{figure}[h]
\includegraphics[scale=0.60]{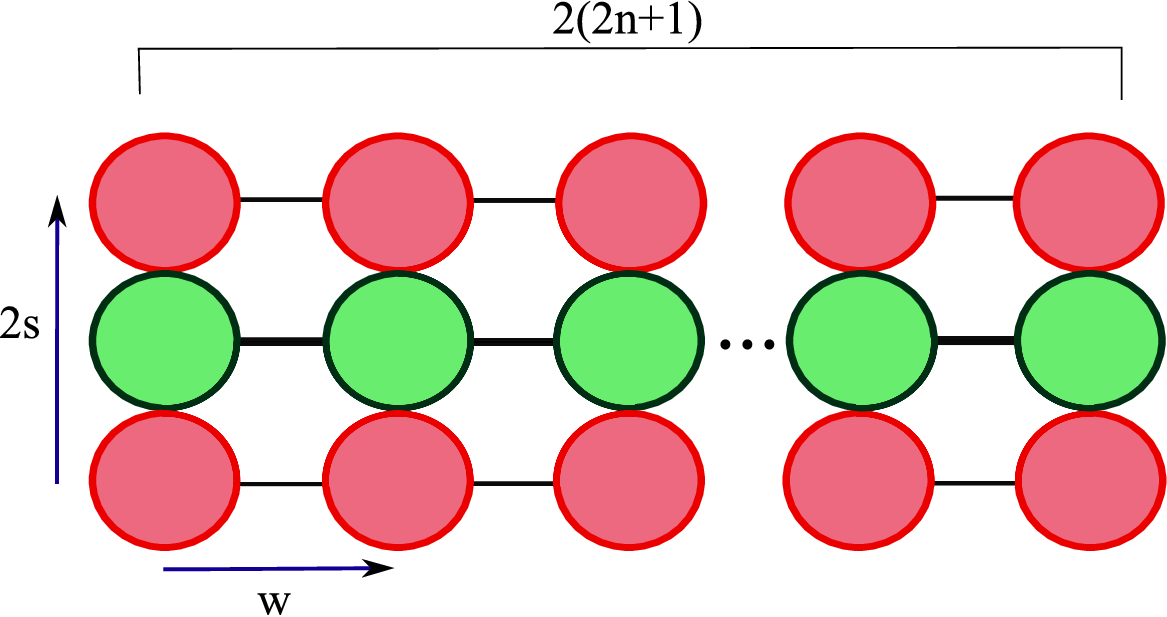}
\caption{The cusp tiling of a cross-section of $C$. The red circles denote the shadows of maximal horoballs from $C$, and the green circles denote the shadows of maximal horoballs from crossing circles.}
\label{cusptiling}
\end{figure}

By Proposition \ref{prop:Knotcusp} and Lemma \ref{lemma:rectanglesize}, the cusp cross-section of $C$ is tiled by a collection of rectangles in a very particular fashion. All of these rectangles have the same dimensions: $\ell(s)$ by $\ell(w)$, with $\ell(s) =1$ and $1 < \ell(w) < 2$.  Furthermore, the circle packing for each of these rectangles is exactly the same. These $4(2n+1)$ rectangles are glued together to form a $2 \times 2(2n+1)$ block of rectangles. Expand this tiling of the cusp cross-section to cover the entire plane.  From our view at $\infty$, we will see the shadow of a maximal horoball centered at each vertex. Specifically, each of the $2n+1$ crossing disks gives three vertices, two of which correspond to horoballs coming from our cusp $C$.  In terms of our $2 \times 2(2n+1)$ block of rectangles, the vertices along the middle row correspond with maximal horoballs of our crossing circles.  Vertices along the top and bottom rows of the block correspond with maximal horoballs from $C$.  We claim that they are in fact the only maximal horoballs of $C$.  See figure \ref{cusptiling} for a diagram showing the maximal horoballs of $C$.

\begin{figure}[h]
\includegraphics[scale=0.60]{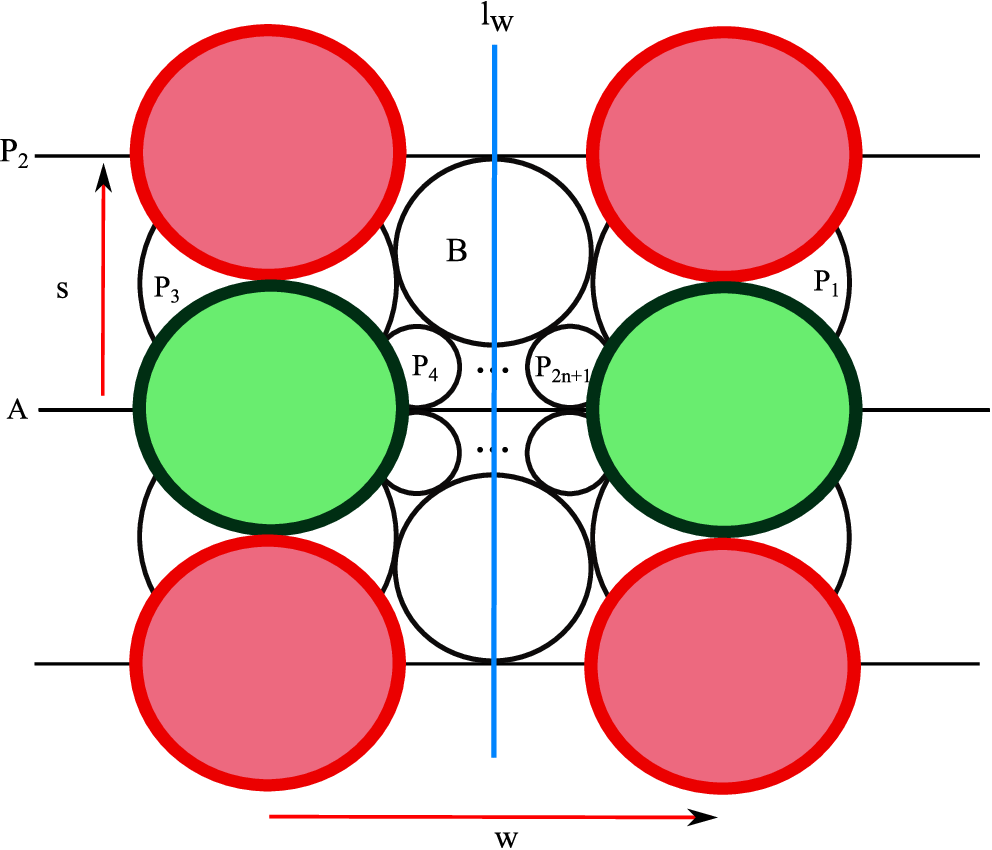}
\caption{The local picture of our cusp tiling of a cross section of $C$. The red circles denote the  shadows of maximal horoballs from $C$, and the green circles denote the shadows of maximal horoballs from crossing circles.}
\label{pretzelcuspshapes4}
\end{figure}

Our circle packing analysis of the rectangles tiling $C$ from Lemma \ref{lemma:rectanglesize} will help us prove this claim. Figure \ref{pretzelcuspshapes4} shows two adjacent rectangles coming from the tiling of $C$, along with their circle packings. This figure also includes the shadows of the maximal horoballs located at vertices. See figure \ref{pretzelcuspshapes3} for a picture of one of these rectangles without the horoball shadows. Suppose there exists another maximal horoball of $C$, call it $H$.  We know $H$ can not intersect the other maximal horoballs, except possibly in points of tangencies. Also, $H$ must be centered at a point either outside of the circles or on the boundary of one of the circles from our circle packing since in constructing our link complement, we cut away hemispheres bound by these circles. On our cusp cross-section of $C$, there are two lines of symmetry that will be useful here:  the line $A$ and the line $l_{w}$, which cuts through the vector $w$ in its midpoints.  Our horoball packing admits reflective symmetries about both of these lines. We shall now consider two cases.

\textbf{Case $1$:} $H$ is centered along $l_{w}$.  Since the center of $H$ can not be contained in $B$, $H$ is either centered at $x_{0} = P_{2} \cap B$ or some $y$ that lies below $B$ and above $A$ on $l_{w}$.  First, suppose $H$ is centered at $x_{0}$. Since $\ell(w) <2$ and there are maximal horoballs at the corners of any such rectangle, $H$ can not be maximal. Now, suppose $H$ is centered at some $y$ as described above.  By applying the reflection along $A$, $H$ will get mapped to another maximal horoball.  For $n \geq 2$, we know that $D(B) > \frac{1}{2}$ by Lemma \ref{lemma:rectanglesize}. Thus, for $n \geq 2$, the distance from the center of $H$ to $l_{w} \cap A$ is less than $\frac{1}{2}$.  In this case, $H$ will overlap with its image.  In order to meet our tangency conditions, $H$ must map to itself.  This implies that $H$ is centered at $y_{0} = l_{w} \cap A$. Once again, since $\ell(w) <2$ and there are maximal horoballs at the corners of any such rectangle, $H$ can not be maximal.

\textbf{Case $2$:} Assume $H$ is not centered along $l_{w}$.  Then the reflection along $l_{w}$ maps $H$ to some other maximal horoball, $H'$.  Now, if $H'$and $H$ intersect, it must be at a point of tangency. So, both $H$ and $H'$ each must be centered a distance at least $\frac{1}{2}$ from $l_{w}$.  This implies that the center of $H$ is at most a distance $\frac{1}{2}$ from the $s$ side of the rectangle closest to $H$. Also, since $\ell(s)=1$, the center of $H$ will be at most a distance $\frac{1}{2}$ from a $w$ side of a rectangle.  Therefore, the center of $H$ will be at most a distance of $\sqrt{(\frac{1}{2})^{2} + (\frac{1}{2})^{2}} = \frac{1}{\sqrt{2}} < 1$ away from a corner of a rectangle, which is also a center of a maximal horoball.  This implies that $H$ will overlap with a maximal horoball at one of the corners, which can not happen. Thus, the only maximal horoballs of $C$ occur at the corners of our rectangles as specified above.

We now claim that the horoball packing corresponding to the cusp $C$ of $N_{2n+1}$ does not admit an order three or order four rotational symmetry.  We fail to have such symmetries because of the shape of our rectangles.  Pick any maximal horoball $H$ of $C$ such that $H \neq H_{\infty}$.  The distance from the center of $H$ to the center of any other maximal horoball of $C$ in the $s$ direction is an integer multiple of $2  \ell(s) = 2$, and the distance from the center of $H$ to the center of any other maximal horoball of $C$ in the $w$ direction is an integer multiple of $\ell(w)$, where $\ell(w) < 2$.   Next, note that the distance across the diagonal of the $2s \times w$ rectangle from the center of $H$ to the center of another maximal horoball of $C$ is $\sqrt{(2\ell(s))^{2}+ \ell(w)^{2}} = \sqrt{4 + \ell(w)^{2}} > \sqrt{5}> 2$ since $\ell(w) > 1$. This implies that the two closest maximal horoballs of $C$ are a distance $\ell(w)$ in the $w$ direction (one to the left of $H$ and one to the right of $H$). This gives an infinite string of pairwise closest maximal horoballs all centered on the same line: take any $H \neq H_{\infty}$ and each translate of $H$ by $n \cdot \ell(w)$, $n \in \mathbb{Z}$ determines another horoball in this string; see figure \ref{cusptiling}. Any rotational symmetry would have to map a string of pairwise closest maximal horoballs to another string of pairwise closest maximal horoballs.  Thus, the only possible rotational symmetry would be an order two symmetry, where each such string maps back to itself. So, the horoball packing of $C$ does not admit an order three or order four rotational symmetry. Thus, this cusp does not cover a $2$-orbifold that has an order three or order four cone point.  But any rigid cusp $2$-orbifold has an order three or order four cone point.  Therefore, $C$ does not cover any rigid cusp $2$-orbifold.

Since the cusp cross-section of $N_{2n+1}$ corresponding to the knot component of $L_{2n+1}$ does not admit order three or order four rotational symmetries, we can now show that the cusp cross-section of $M$ also doesn't have these symmetries.  This is made possible by taking sufficiently long Dehn fillings along the crossing circles. As $q_{i} \rightarrow \infty$, any such $M$ converges to $N_{2n+1}$ in both the geometric topology and the algebraic topology. This convergence implies that we can fix a compact subset of $\mathbb{H}^{3}$, and the geometry of our horoball packing of $C'$ (the cusp of $M$ corresponding to the knot $K$) can be made sufficiently close to the geometry of $C$ on this compact subset, by choosing $q_{i}$ sufficiently large. So, consider the set of maximal horoballs $H_{1}, \dots, H_{k}$ of $C$ that intersect some fixed fundamental domain for the stabilizer of $\infty$. We claim that for sufficiently small $\delta$, we can choose $q_{i}$ large enough so that each such $H_{j}$ has radius and center $\delta$-close to a corresponding horoball $H_{j}'$ in $C'$.

Let $g_{j}$ be a deck transformation of $N_{2n+1}$ with $H_{j} = g_{j}(H_{\infty})$ (for $1 \leq j \leq k$). Then for each $M = \mathbb{S}^{3} \setminus K = N_{2n+1}\left( (1, q_{1}) , \dots, (1, q_{2n+1}) \right)$, we obtain a covering transformation $g_{j_{i}}$  with $g_{j_{i}} \rightarrow g_{j}$ as $i \rightarrow \infty$ in the algebraic topology (here, $i \rightarrow \infty$ means that all $2n+1$ Dehn surgery coefficients are heading to infinity). This convergence implies that the centers and radii of $H_{j}' = g_{j_{i}}(H_{\infty})$ approach the center and radii of $H_{j}$, respectively. This gives us the desired set of horoballs $H_{1}' ,\dots, H'_{k}$ in $C'$, which we refer to as \textit{almost maximal horoballs}.   

Now we can show that $C'$ lacks any order three or order four rotational symmetries by using the same type of argument we used for $C$. For $C$, we had infinite strings of pairwise closest maximal horoballs, with each string centered on a horizontal line. For $C'$, we get finite strings (since we are working over a compact domain) of pairwise closest almost maximal horoballs. These horoballs might not be centered on horizontal lines, but instead, are within a sufficiently small $\epsilon$ of being centered on horizontal lines. If anything, this will only further break any possible symmetries. Any rotational symmetry would have to map a string of pairwise closest almost maximal horoballs to another string of pairwise closest almost maximal horoballs. Again, the only possible rotational symmetry would be an order two symmetry. Thus, the one cusp of $M$ cannot cover a rigid $2$-orbifold, and so, $M$ does not admit hidden symmetries. 
\end{proof}

Combining Proposition \ref{prop:no_rigid} with Proposition \ref{prop:no_surg_or_sym} shows that we have covered the three criteria in Reid and Walsh's theorem. This gives the following theorem, which applies to our pretzel knots $K_{2n+1} = K \left( \frac{1}{q_{1}}, \frac{1}{q_{2}}, \ldots, \frac{1}{q_{2n+1}} \right)$, if we assume that all $q_{i}$ are sufficiently large.

\begin{thm}
\label{thm:non_commensurable}
Let $n \geq 2$ and let $q_{1}, \ldots, q_{2n+1}$ be integers such that only $q_{1}$ is even, $q_{i} \neq q_{j}$ for $i \neq j$, and all $q_{i}$ are sufficiently large. Then the complement of the hyperbolic pretzel knot  $K \left( \frac{1}{q_{1}}, \frac{1}{q_{2}}, \ldots, \frac{1}{q_{2n+1}} \right)$  is the only knot complement in its commensurability class.  In particular, any two of these hyperbolic pretzel knot complements are incommensurable.
\end{thm}

The work of Schwartz \cite[Theorem 1.1]{Sch} tells us that two cusped hyperbolic $3$-manifolds are commensurable if and only if their fundamental groups are quasi-isometric.  This immediately gives the following corollary.

\begin{cor}
\label{cor:notQI}
If two pretzel knot complements as described in Theorem \ref{thm:non_commensurable} are non-isometric, then they do not have quasi-isometric fundamental groups.  
\end{cor} 

\textbf{Remark:} The work of Goodman--Heard--Hodgson \cite{GoHeHo} implies that two hyperbolic knot complements are commensurable if and only if there exist horoball neighborhoods (of each knot complement) that lift to isometric packings of $\mathbb{H}^3$. This could provide another method to prove that any pair of our pretzel knot complements that differ by a composition of mutations are incommensurable (assuming they are non-isometric): show that their corresponding horoball packings are non-isometric for \textit{any} possible horoball neighborhoods. We could again try to use our horoball packing coming from the cusp $C$ of $N_{2n+1}$ to analyze the horoball packings of each $\mathbb{S}^{3} \setminus K_{2n+1}^{\sigma}$. However, to conclude that two of our knot complements are incommensurable, we would need to vary our cusp neighborhoods rather than just work with the canonical choice. Also, this proof technique would not imply that $\mathbb{S}^{3} \setminus K_{2n+1}^{\sigma}$ and $\mathbb{S}^{3} \setminus K_{2m+1}^{\sigma}$ are incommensurable when $n \neq m$, and so, we do need Proposition \ref{prop:no_rigid} for the stronger statement that any two of our pretzel knot complements are incommensurable.

%%%%%%%%%%%%%%%%%%%%%%%%%%%%%%%%%%%%%%%%%%%%%%%%%%%%%%%%%%%%

\section{Mutations and short geodesics coming from Dehn fillings}
\label{sec:mutations_sys}
In this section, we shall analyze the behavior of short geodesics in the set of knot complements $\left\{M_{2n+1}^{\sigma}\right\}$. If there is enough vertical twisting in each twist region, i.e., if each $q_{i}$ is sufficiently large, then we can easily figure out which geodesic are the shortest.  This analysis is possible by realizing our pretzel knot complements as Dehn surgeries along untwisted augmented link complements.  We shall also see that if each $q_{i}$ is sufficiently large, then the initial length spectrum is actually preserved under mutation, and so, we will be able to generate a large class of hyperbolic knot complements with both the same volume and the same initial length spectrum. Here, we also give an application to closed hyperbolic $3$-manifolds that come from Dehn filling $M_{2n+1}^{\sigma}$ along $K_{2n+1}^{\sigma}$. For each $n \in \mathbb{N}$, $n \geq 2$ these sets of closed manifolds will have the same volume and the same initial length spectrum.  We end this section by raising some questions about the effectiveness of geometric invariants of hyperbolic $3$-manifolds.

\subsection{Mutations of $K_{2n+1}$ with the same initial length spectrum}
\label{subsec:mutations_of_K}

Given the untwisted augmented link complement $N_{2n+1} = \mathbb{S}^{3} \setminus L_{2n+1}$, we  form $M_{2n+1} =\mathbb{S}^{3} \setminus K_{2n+1}$ by performing Dehn surgeries $(1, \frac{q_{i} - 1}{2})$ along $2n$ of the crossing circle cusps, and one Dehn surgery $(1 , \frac{q_{1}}{2})$  along the crossing circle cusp not enclosing a half-twist, i.e., 
\begin{center}
$M_{2n+1} = N_{2n+1} \left( (1, \frac{q_{1}}{2}), (1,  \frac{q_{2} - 1}{2}), \dots, (1,  \frac{q_{2n+1} - 1}{2}) \right)$.  
\end{center}
Similarly, any mutation $M^{\sigma}_{2n+1}$ is obtained by performing the same Dehn surgeries on $N_{2n+1}$, just with some of the surgery coefficients permuted. We now show that if each $q_{i}$ is sufficiently large, then the core geodesics in $M_{2n+1}$ are sufficiently short, and so, they are preserved under mutation.

\begin{thm}
\label{thm:sys_preserved}
Let $\left\{ \gamma_{i}^{\sigma} \right\}_{i=1}^{2n+1}$ be the $2n+1$ geodesics in $M_{2n+1}^{\sigma}$ that came from Dehn filling the crossing circles of $N_{2n+1}$.  For each $n \in \mathbb{N}$, there exists a constant $Q = Q(n) = \sqrt{(20.76)^{2}\frac{(2n+1)(4n)}{2n-1}-1}$, such that if each $q_{i} \geq Q$, then $\left\{ \gamma_{i}^{\sigma} \right\}_{i=1}^{2n+1}$ make up at least $2n+1$ of the shortest geodesics in their respective hyperbolic $3$-manifold and every $M_{2n+1}^{\sigma}$ has at least the same shortest $2n+1$ (complex) geodesic lengths.
\end{thm}

\begin{proof}
Given $M_{2n+1}$, we must show that the result holds for a mutation $\sigma_{a}$ along $S_{a}$, and the general result will follow by induction. By Proposition \ref{prop:NL}, we know that the normalized length of the $i^{th}$ filling slope satisfies $\widehat{L}\left(s_{i}\right) \geq  \sqrt{\frac{(2n-1)(1+q_{i}^{2})}{4n}}$. If  each $\widehat{L}(s_{i}) \geq 20.76\sqrt{2n+1}$, then Corollary \ref{cor:syspreservednl} tells us that $M$ and $M^{\sigma_{a}}$ have (at least) the same $2n+1$ shortest (complex) geodesic lengths, and (at least) a portion of the initial length spectrum is given by $\left\{\ell_{\mathbb{C}}(\gamma_{i})\right\}_{i=1}^{2n+1} = \left\{\ell_{\mathbb{C}}(\gamma_{i}^{\sigma_{a}})\right\}_{i=1}^{2n+1}$. Thus, we need to just solve  $\sqrt{\frac{(2n-1)(1+q_{i}^{2})}{4n}} \geq 20.76\sqrt{2n+1}$ for $q_{i}$ to determine $Q$.
\end{proof}

The following theorem comes from combining Theorem \ref{thm:sys_preserved}, Theorem \ref{thm:non_commensurable}, and Theorem \ref{thm:volume}, and requires all $q_{i}$ to be chosen sufficienty large. This theorem shows that there are large classes of geometrically similar pretzel knots -- they have non-isometric knot complements, but a large number of their geometric invariants are the same.

\begin{thm}
\label{thm:similarpretzels}
For each $n \in \mathbb{N}$, $n\geq 2$, there exist $\frac{(2n)!}{2}$ non-isometric hyperbolic pretzel knot complements, $\left\{M_{2n+1}^{\sigma}\right\}$, such that these manifolds:

\begin{itemize}
\item have the same $2n+1$ shortest geodesic (complex) lengths, 
\item are pairwise incommensurable,
\item have the same volume, and
\item $\left(\frac{2n-1}{2}\right)v_{\mathrm{oct}} \leq vol(M^{\sigma}_{2n+1})  \leq \left(4n+2\right)v_{\mathrm{oct}}$, where $v_{\mathrm{oct}} \left(\approx 3.6638\right)$ is the volume of a regular ideal octahedron. 

\end{itemize}
\end{thm}

%%%%%%%%%%%%%%%%%%%%%%%%%%%%%%%%%%%%%%%%%%%%%%%%%%%%%%%%%%%%%%%

\subsection{Closed hyperbolic $3$-manifolds with the same volume and initial length spectrum}
\label{subsec:closed}

Let $M = \mathbb{S}^{3} \setminus K$ and let $M(p,q)$ denote the closed manifold obtained by performing a $(p,q)$-Dehn surgery along the knot $K$.  In \cite[Theorem 3]{Mi}, we show that for each $n \in \mathbb{N}$, $n \geq 2$, and for $(p,q)$ sufficiently large, $M_{2n+1}^{\sigma}(p,q)$ and $M_{2n+1}^{\sigma'}(p,q)$ have the same volume and are non-isometric closed hyperbolic $3$-manifolds, whenever $M_{2n+1}^{\sigma}$ and $M_{2n+1}^{\sigma'}$ are non-isometric.  This proof relies on another result of Ruberman's \cite[Theorem 5.5]{Ru} which shows that corresponding Dehn surgeries on a hyperbolic knot $K$ and its fellow mutant $K^{\mu}$ will often result in manifolds with the same volume.  Specifically, this happens when a Conway sphere and its mutation are \textit{unlinked}.

\begin{defn}[Unlinked]
\label{def:Unlinked}
Let $K$ be a knot in $S^{3}$ admitting a Conway sphere $S$.  Observe that a specific choice of a mutation $\mu$ gives a pair of $S^{0}$'s on the knot such that each $S^{0}$ is preserved by $\mu$. We say that $\mu$ and $S$ are \textit{unlinked} if these $S^{0}$'s are unlinked on $K$.  
\end{defn}

Being unlinked allows one to tube together the boundary components of a Conway sphere that are interchanged by $\mu$ to create a closed surface of genus two, which we call $S'$. $S'$ is also a hyperelliptic surface, and its involution is the same as the involution $\mu$ of our Conway sphere. Dehn surgeries on $\mathbb{S}^{3} \setminus K$ and its mutant $\mathbb{S}^{3} \setminus K^{\mu}$ differ by mutating along this closed surface. Thus, Ruberman's result for preserving volume will apply to these closed manifolds.

Combining our work in \cite{Mi} with Corollary \ref{cor:syspreserved2} gives the following.

\begin{thm}
\label{thm:closedmanifolds}
For each $n \in \mathbb{N}$, $n \geq 2$, and any $(p,q)$ sufficiently large, there exist $\frac{(2n-1)!}{2}$ non-isometric closed hyperbolic $3$-manifolds $\left\{M_{2n+1}^{\sigma}(p,q)\right\}$ such that these manifolds:
\begin{itemize}
\item have the same $2n+2$ shortest (complex) geodesic lengths,
\item have the same volumes, and
\item $vol(M^{\sigma}_{2n+1}(p,q)) < (4n+2)v_{oct}$.
\end{itemize}

\end{thm}

\begin{proof}
In \cite{Mi}, we constructed our $K_{2n+1}$ so that all Conway spheres in $\left\{ (S_{a}, \sigma_{a}) \right\}_{a=1}^{2n}$ are unlinked.  However, here we have slightly modified this construction of each $K_{2n+1}$.  Specifically, we now have one twist region with an even number of twists in $K_{2n+1}$.  As a result, $(S_{1}, \sigma_{1})$ is not unlinked. Thus, we will only mutate along the other Conway spheres: $\left\{ (S_{a}, \sigma_{a}) \right\}_{a=2}^{2n}$.  These combinations of mutations create $\frac{(2n)!}{2(2n)} = \frac{(2n-1)!}{2}$ non-isometric, hyperbolic pretzel knots; see \cite[Theorem $2$]{Mi} for more details.  

Let $\sigma$ and $\sigma'$ be any combination of mutations along our unlinked Conway spheres resulting in non-isometric knot complements.  Now, $M_{2n+1}^{\sigma}(p,q)$ and $M_{2n+1}^{\sigma'}(p,q)$ have the same volume by Ruberman's work.  In \cite[Theorem $3$]{Mi}, we show that $M_{2n+1}^{\sigma}(p,q)$ and $M_{2n+1}^{\sigma'}(p,q)$ are non-isometric by choosing $(p,q)$ sufficiently large so that the core geodesics resulting from this Dehn filling are the systoles of their respective manifolds.  This comes from the work of Neumann--Zagier \cite{NZ}. So, for $(p,q)$ sufficiently large, any $M_{2n+1}^{\sigma}(p,q)$ will have at least $2n+2$ closed geodesics shorter than a constant $L< 0.015$. $2n+1$ of these geodesics come from Dehn filling our crossing circles of $L_{2n+1}$, and the systole comes from then Dehn filling the knot component. We can apply Corollary \ref{cor:syspreservednl} to these closed manifolds to show that they have the same $2n+2$ shortest geodesic lengths.  The upper bound on volume follows from the proof of \cite[Theorem $3$]{Mi}.
\end{proof}

\subsection{Closing Remarks}
\label{subsec:CR}
The fact that the manifolds $\left\{M_{2n+1}^{\sigma}\right\}$ are constructed by mutating knot complements that are pairwise incommensurable sharply contrasts any of the known constructions for building large classes of hyperbolic $3$-manifolds that are iso-length spectral.  However, we only know that our mutant knot complements have the same initial length spectra. Based on experimental evidence from SnapPy, the author doubts that any of these manifolds actually are iso-length spectral. It would be interesting to know if this mutation process could be used to produce iso-length spectral hyperbolic $3$-manifolds that are incommensurable. 

In addition, there is a general recipe for our type of construction and we did not necessarily need to use pretzel knots.  In order to construct a large number of non-isometric hyperbolic manifolds with the same volume and the same initial length spectrum, you need the following key ingredients.

\begin{itemize}
\item An initial hyperbolic $3$-manifold $M$ with:
\begin{itemize}
\item a large number of hyperelliptic surfaces in $M$ to mutate along to create the set of manifolds $\left\{M^{\sigma}\right\}$, and
\item a way to determine your shortest geodesics in $M$ and make sure they are sufficiently short, i.e., realize them as the cores of sufficiently long Dehn fillings.
\end{itemize}
\item A simple method to determine how much double counting you are doing, i.e., a method to determine if any $M^{\sigma}$ and $M^{\sigma'}$ are isometric or not.
\end{itemize}

Given this recipe, you want to maximize the number of hyperelliptic surfaces in $M$ to mutate along and maximize the number of sufficiently short geodesics, while minimizing the double counting. It would be interesting to examine how well we did with maximizing and minimizing these parameters. Such an examination leads us to consider the function $N(v,s)$, which counts the number of hyperbolic $3$-manifolds with same volume $v$ and the same $s$ shortest geodesic lengths.  We can also consider the restriction of this counting function to specific classes of hyperbolic $3$-manifolds. Let $N_{K}(v,s)$ denote the restriction of $N(v,s)$ to hyperbolic knot complements and $N_{Cl}(v,s)$ denote the restriction of $N(v,s)$ to closed hyperbolic $3$-manifolds. An immediate corollary of Theorem \ref{thm:similarpretzels} and Theorem \ref{thm:closedmanifolds} gives the following lower bounds on the growth rates of $N_{K}(v,s)$ and $N_{Cl}(v,s)$ as functions of $v$. The proof of this corollary is the same as the proof of \cite[Theorem $1$]{Mi}, except we can now take the short geodesic lengths into account. 

\begin{cor}
\label{cor:growth}
There are sequences $\left\{(v_{n}, s_{n})\right\}$ and $\left\{(x_{n}, t_{n})\right\}$ with $(v_{n}, s_{n}),(x_{n}, t_{n}) \rightarrow (\infty, \infty)$ such that 
\begin{center}
$N_{K}((v_{n}, s_{n})) \geq (v_{n})^{(\frac{v_{n}}{8})}$ and $N_{Cl}((x_{n}, t_{n})) \geq \left(x_{n} \right)^{\left(\frac{x_{n}}{8} \right)}$
\end{center}
for all $n \gg 0$. 
\end{cor}

This corollary tells us that the counting function $N((v,s))$ grows at least factorially fast with $v$, and immediately raises some questions. 
\begin{question}
Can a Sunada-type construction or an arithmetic method be applied to also show $N((v,s))$ grows at least factorially fast with $v$? Also, are there sequences $\left\{(v_{n}, s_{n})\right\}$ with $v_{n} \rightarrow \infty$ such that $N((v_{n}, s_{n}))$ grows faster than factorially with $v_{n}$? 
\end{question}
It would be interesting to find a construction realizing a growth rate faster than the one given in Corollary \ref{cor:growth} or show that a factorial growth rate is actually the best we can do.

%%%%%%%%%%%%%%%%%%%%%%%%%%%%%%%%%%%%%%%%%%%%%%%%%%%%%%%

\bibliographystyle{hamsplain}
\bibliography{systolepaperbiblio}

\end{document}